\documentclass[a4paper]{article}

\usepackage[top=1.5in,right=1in,left=1in,bottom=1in,a4paper]{geometry}
\usepackage{amsmath, amsthm, amssymb, amsfonts}
\usepackage{enumerate}
\usepackage{graphicx}
\usepackage[all]{xy}

\newtheoremstyle{theoremstyle}
  {10pt}      
  {5pt}       
  {\itshape}  
  {}          
  {\bfseries} 
  {:}         
  {.5em}      
  {}          

\newtheoremstyle{examplestyle}
  {10pt}      
  {5pt}       
  {}          
  {}          
  {\bfseries} 
  {:}         
  {.5em}      
  {}          

\theoremstyle{theoremstyle}
\newtheorem{theorem}{Theorem}[section]
\newtheorem*{theorem*}{Theorem}
\newtheorem{lemma}[theorem]{Lemma}
\newtheorem{proposition}[theorem]{Proposition}
\newtheorem*{proposition*}{Proposition}
\newtheorem{corollary}[theorem]{Corollary}
\newtheorem*{corollary*}{Corollary}

\theoremstyle{examplestyle}
\newtheorem{example}[theorem]{Example}
\newtheorem{definition}[theorem]{Definition}
\newtheorem*{definition*}{Definition}
\newtheorem{remark}[theorem]{Remark}
\newtheorem*{remarks*}{Remarks}

\newtheorem*{remark*}{Remark}

\newtheorem*{conjecture*}{Conjecture}
\newtheorem{convention}[theorem]{Convention}

\newcommand{\comment}[1]{}

\newcommand{\pic}{\operatorname{Pic}}

\newcommand{\sh}[1]{\mathcal{#1}}

\newcommand{\spec}{\operatorname{spec}}

\newcommand{\rk}{\operatorname{rk}}

\newcommand{\shHom}{\mathcal{H}om}
\newcommand{\Ext}{\operatorname{Ext}}
\newcommand{\shExt}{\mathcal{E}xt}

\newcommand{\stern}{{\operatorname{star}}}
\newcommand{\circuit}{{\mathcal{C}}}
\newcommand{\ocircuit}{{\mathfrak{C}}}

\newcommand{\lcm}{\operatorname{lcm}}

\newcommand{\Z}{\mathbb{Z}}
\newcommand{\Q}{\mathbb{Q}}
\newcommand{\R}{\mathbb{R}}
\newcommand{\N}{\mathbb{N}}
\newcommand{\oo}{\mathbb{O}}

\newcommand{\uc}{{\underline{c}}}
\newcommand{\on}{{[n]}}

\newcommand{\secfan}{\operatorname{SF}}

\newcommand{\nef}{{\operatorname{nef}}}

\begin{document}

\title{Divisorial cohomology vanishing on toric varieties}

\author{Markus Perling\footnote{Ruhr-Universit\"at Bochum, Fakult\"at f\"ur Mathematik,
Universit\"atsstra\ss e 150, 44780 Bochum, Germany, {\tt Markus.Perling@rub.de}}}

\date{August 2009}

\maketitle

\begin{abstract}
This work discusses combinatorial and arithmetic aspects of cohomology vanishing for
divisorial sheaves on toric varieties. We obtain a
refined variant of the Kawamata-Viehweg theorem which is slightly stronger. Moreover, we
prove a new vanishing theorem related to divisors whose inverse is nef and has small
Iitaka dimension.
Finally, we give a new criterion for divisorial sheaves for being maximal Cohen-Macaulay.
\end{abstract}

\tableofcontents

\section{Introduction}

The goal of this paper is to more thoroughly understand cohomology vanishing for
divisorial sheaves on toric varieties. The motivation for this comes from the
calculations of \cite{HillePerling06}, where a counterexample to a conjecture
of King \cite{King2} concerning the derived category $D^b(X)$ of smooth complete
toric varieties was presented.
Based on work of Bondal (see \cite{Rudakov90}, \cite{Bondal90}), it was conjectured
that on every smooth complete toric variety $X$
there exists a full strongly exceptional collection of line bundles. That is,
a collection of line bundles $\sh{L}_1, \dots, \sh{L}_n$ on $X$ which generates
$D^b(X)$ and has the property that $\Ext^k(\sh{L}_i, \sh{L}_j) = 0$ for all $k > 0$
and all $i, j$. 
Such a collection induces an equivalence of categories
$\operatorname{RHom}(\bigoplus_i \mathcal{L}_i, \, . \,): D^b(X) \longrightarrow
D^b\big(\operatorname{End}(\bigoplus_i \mathcal{L}_i)-\operatorname{mod} \big)$.
This possible generalization of Beilinson's theorem (pending the existence of a full
strongly exceptional collection) has attracted much interest, notably also in the
context of the homological mirror conjecture \cite{Kontsevich94}.
For line bundles, the problem of $\operatorname{Ext}$-vanishing can be
reformulated to a problem of cohomology vanishing for line bundles by the
isomorphisms
\begin{equation*}
\operatorname{Ext}^k(\mathcal{L}_i, \mathcal{L}_j) \cong H^k(X, \mathcal{L}_i\check{\ }
\otimes \mathcal{L}_j) = 0 \text{ for all } k \geq 0 \text{ and all } i, j.
\end{equation*}
So we are
facing a quite peculiar cohomology vanishing problem: let $n$ denote the rank of the
Grothendieck group of $X$, then we look for a certain constellation
of $n (n - 1)$ -- not necessarily distinct -- line bundles, all of which have vanishing higher cohomology groups.
The strongest general vanishing theorems so far are of the Kawamata-Viehweg type
(see \cite{Mustata1} and \cite{Fujino07}, and also \cite{Materov02} for 
Bott type formulas for cohomologies of line bundles), but it can be seen
from very easy examples, such as Hirzebruch surfaces, that these alone in general do
not suffice
to prove or disprove the existence of strongly exceptional collections by means
of cohomology vanishing. In \cite{HillePerling06}, on a certain toric surface $X$,
all line bundles $\sh{L}$ with the property that $H^i(X, \sh{L}) =
H^i(X, \sh{L}\check{\ }) = 0$ for all $i > 0$ were completely classified by
making use of a explicit toric representation of the cohomology vanishing problem
for line bundles. This approach exhibits quite complicated combinatorial as well
as number theoretic conditions for cohomology vanishing which we are going to describe
in general.

We will consider and partially answer the following more general problem. Let $D$ be a
Weil divisor on any toric variety $X$ and $V \subset X$ a torus invariant closed subscheme,
then what are necessary and sufficient conditions for the (global) local cohomology
modules $H^i_V\big(X, \sh{O}_X(D)\big)$ to vanish? Given this spectrum of cohomology
vanishing problems, we have on one extreme end the cohomology vanishing problem for
line bundles, and on another extreme end the classification problem for maximal
Cohen Macaulay (MCM) modules over semigroup rings: on an affine toric variety $X$,
the sheaf $\sh{O}_X(D)$ is MCM if and only if the local cohomologies
$H^i_x\big(X, \sh{O}_X(D)\big)$ vanish for $i \neq \dim X$, where $x \in X$ is the torus
fixed point.
These local cohomologies have been studied by Stanley \cite{Stanley82}, \cite{Stanley96}
(see also \cite{vandenbergh92} for generalizations), and Bruns and Gubeladze \cite{BrunsGubeladze}
showed that only finitely many sheaves in this class are MCM. MCM sheaves over affine toric
varieties have only been classified for some special cases
(see, for instance \cite{BuchweitzGreuelSchreyer}
and chapter 16 of \cite{Yoshino90}). Our contribution will
be to give a more explicit combinatorial characterization of MCM modules of rank one
over normal semigroup rings and their ties to the birational geometry of toric varieties.

One important aspect of our results is that, though we will also make use of $\Q$-divisors,
our vanishing results will completely be formulated in the integral setting. We will
illustrate the effect of this by the following
example. Consider the weighted projective surface $\mathbb{P}(2, 3, 5)$. Then the divisor
class group $A_1\big(\mathbb{P}(2, 3, 5)\big)$ is isomorphic to $\Z$ and, after fixing
the generator $D = 1$ of $A_1\big(\mathbb{P}(2, 3, 5)\big)$ to be $\Q$-effective, the torus invariant
irreducible divisors can be identified with the integers $2$, $3$, and $5$, and the canonical
divisor has class $-10$. By Kawamata-Viehweg we obtain that $H^2\big((\mathbb{P}(2, 3, 5),
\sh{O}(kD)\big) = 0$ for $k > -10$. However, as we will explain in more detail below, the set
of all divisors $kD$ with nontrivial second cohomology is given by all $k$ with $-k = 2r + 3s
+ 5t$ with $r, s, t$ positive integers. So, Kawamata-Viehweg misses the divisor $-11 D$.
The reason is that the toric Kawamata-Viehweg vanishing theorem tells us that the cohomology
of some divisor $D'$ vanishes if the rational equivalence class over $\Q$ of $D' -
K_{\mathbb{P}(2, 3, 5)}$ is
contained in the interior of the nef cone in $A_1\big(\mathbb{P}(2, 3, 5)\big)_\Q$. Over the
integers, the domain of cohomology vanishing thus in general is larger than over $\Q$. Below we will
see that this is a general feature of cohomology vanishing, even for smooth toric varieties,
as can be seen, for instance, by considering the strict transform of the divisor $-11D$ along some
toric blow-up $X \longrightarrow \mathbb{P}(2, 3, 5)$ such that $X$ is smooth.

\paragraph{The main results.}

The first main result will be an integral version of the Kawamata-Viehweg
vanishing theorem. Consider the nef cone $\nef(X) \subset A_{d - 1}(X)_\Q$, then
the toric Kawamata-Viehweg vanishing theorem (see Theorem \ref{KodairaKawamataViehweg})
can be interpreted such that if $D - K_X$ is contained in the interior of $\nef(X)$,
then $H^i\big(X, \sh{O}_X(D)\big) = 0$ for all $i > 0$. For our version we will
define a set $\mathfrak{A}_\nef \subset A_{d - 1}(X)$, which we call the
{\em arithmetic core} of $\nef(X)$ (see definition \ref{vanishingcoredefinition}).
The set $\mathfrak{A}_\nef$ has the
property that it contains all integral Weil divisors which map to the interior of
the cone $K_X + \nef(X)$ in $A_{d - 1}(X)_\Q$. But in general it is strictly larger,
as in the example above. We can lift the cohomology vanishing theorem for divisors
in $\nef(X)$ to $\mathfrak{A}_\nef$:

\begin{theorem*}[\ref{nefvanishing}]
Let $X$ be a complete toric variety and $D \in \mathfrak{A}_\nef$.
Then $H^i\big(X, \mathcal{O}_X(D)\big)$ $= 0$ for all $i > 0$.
\end{theorem*}

One can consider Theorem \ref{nefvanishing} as an ``augmentation'' of the standard
vanishing theorem for nef divisors to the subset $\mathfrak{A}_\nef$
of $A_{d - 1}(X)$. In general,
Theorem \ref{nefvanishing} is slightly stronger than the toric Kawamata-Viehweg vanishing
theorem and yields refined arithmetic conditions.

However, the  main goal of this paper is to find vanishing results which cannot
directly be derived from known vanishing theorems.
Let $D$ be a nef Cartier divisor whose Iitaka dimension
is positive but smaller than $d$. This class of divisors is contained in nonzero faces
of the nef cone of $X$ which are contained in the intersection of the nef cone with
the boundary of the effective cone of $X$ (see Section \ref{nonstandardvanishing}).
Let $F$ be such a face. Similarly as with $\mathfrak{A}_\nef$, we can define for the inverse
cone $-F$ an arithmetic core $\mathfrak{A}_{-F}$ (see \ref{vanishingcoredefinition}) and
associate to it a
vanishing theorem, which may be considered as the principal result of this article:

\begin{theorem*}[\ref{projectiveantinef}]
Let $X$ be a complete $d$-dimensional toric variety. Then $H^i\big(X, \sh{O}(D)\big) = 0$
for every $i$ and all $D$ which are contained in some $\mathfrak{A}_{-F}$,
where $F$ is a face of  $\nef(X)$ which
cointains nef divisors of Iitaka dimension
$0 < \kappa(D) < d$. If $\mathfrak{A}_{-F}$ is nonempty, then it contains infinitely many
divisor classes.
\end{theorem*}

This theorem cannot be an augmentation of a vanishing theorem for $-F$, as it is
not true in general that $H^i\big(X, \sh{O}_X(-D)\big) = 0$ for all $i$ for
$D$ nef of Iitaka dimension smaller than $d$. In particular, the set of
$\Q$-equivalence classes of elements in $\mathfrak{A}_{-F}$ does not intersect
$-F$.

For the case of a toric surface $X$ we show that  above vanishing theorems
combine to a nearly complete vanishing theorem for $X$. Recall that in the fan
associated to a complete toric surface $X$ every pair of opposite rays by projection
gives rise to a morphism from $X$ to $\mathbb{P}^1$ (e.g. such a pair does always
exist if $X$ is smooth and $X \neq  \mathbb{P}^2$). Correspondingly, we
obtain a family of nef divisors of Iitaka dimension $1$ on $X$ given by the
pullbacks of the sheaves $\sh{O}_{\mathbb{P}_1}(i)$ for $i > 0$. We get:

\begin{theorem*}[\ref{surfaceclassification}]
Let $X$ be a complete toric surface. Then there are only finitely many divisors
$D$ with $H^i\big(X, \sh{O}_X(D)\big) = 0$ for all $i > 0$ which are not contained
in $\mathfrak{A}_\nef \cup \bigcup_F \mathfrak{A}_{-F}$, where the union
ranges over all faces of $\nef(X)$ which correspond to pairs of opposite rays
in the fan associated to $X$.
\end{theorem*}

Some more precise numerical characterizations on the sets $\mathfrak{A}_{-F}$ will
be given in subsection \ref{nonstandardvanishing}.
The final result is a birational characterization of MCM-sheaves of rank one.
This is a test case to see whether point of view of birational geometry might
be useful for classifying more general MCM-sheaves. The idea for this comes from the
investigation of
MCM-sheaves over surface singularities in terms of resolutions in the context of
the McKay correspondence
(see \cite{GonzalezSprinbergVerdier}, \cite{ArtinVerdier},
\cite{EsnaultKnoerrer}). For an affine toric variety $X$, in general one cannot
expect to find a similar nice correspondence. However, there is a set of
preferred partial resolutions of singularities $\pi: \tilde{X} \longrightarrow X$
which is parameterized by the secondary fan of $X$. Our result is a toric analog of
a technical criterion of loc. cit.

\begin{theorem*}[\ref{mcmtheorem}]
Let $X$ be a $d$-dimensional affine toric variety whose associated cone has simplicial
facets and let $D \in A_{d - 1}(X)$. If
$R^i\pi_* \sh{O}_{\tilde{X}}(\pi^*D) = 0$ for every regular triangulation
$\pi: \tilde{X} \longrightarrow X$, then $\sh{O}_X(D)$ is MCM. For $d = 3$
the converse is also true.
\end{theorem*}

Note that the facets of a $3$-dimensional cone are always simplicial.

To prove our results we will require a lot of bookkeeping, combining various geometric,
combinatorial and arithmetic aspects of toric varieties. This has the unfortunate effect
that the exposition will be rather technical
and incorporate many notions (though not much theory) coming from combinatorics.
As this might be cumbersome to follow for the more geometrically inclined reader, we
will give an overview of the key structures and explain how they fit together.
From now $X$ denotes an arbitrary $d$-dimensional toric variety, $\Delta$ the fan
associated to $X$, $M \cong \Z^d$ the character group of the torus which acts
on $X$. We denote $\on := \{l_1, \dots, l_n\}$ the set of primitive vectors of the
$1$-dimensional cones in $\Delta$ and $D_1, \dots, D_n$ the corresponding
torus invariant prime divisors on $X$.

\paragraph{Cohomology and simplicial complexes.}

We will follow the standard approach for computing cohomology of torus invariant
Weil divisors, using the induced eigenspace decomposition. Let $D$ be such a divisor
on $X$ and $V \subseteq X$ a torus-invariant subscheme, then:
\begin{equation*}
H^i_V\big(X, \sh{O}_X(D)\big) \cong \bigoplus_{m \in M} H^i_V\big(X, \sh{O}_X(D)\big)_m.
\end{equation*}
We denote $\hat{\Delta}$ the simplicial model of $\Delta$, i.e. the abstract
simplicial complex on the set $\on$ such that any subset $I \subset \on$ is in
$\hat{\Delta}$ iff there
exists a cone $\sigma$ in $\Delta$ such that elements in $I$ are faces of
$\sigma$. Similarly,
we define a subcomplex $\hat{\Delta}_V$, by considering only those cones in $\Delta$
whose associated orbits in $X$ are not contained in $V$ (see also Section \ref{prelim}).
For a given torus invariant divisor $D = \sum_{i = 1}^n c_i D_i$ and any character
$m \in M$, we set $\hat{\Delta}_m$ and $\hat{\Delta}_{V, m}$ the
full subcomplexes which are supported on those $l_i$ with $l_i(m) < -c_i$. Now
the general formula for cohomology of $\sh{O}_X(D)$ is given as the relative
reduced cohomology of the complexes $\hat{\Delta}$ and $\hat{\Delta}_V$ with coefficients
in our base field $k$:

\begin{theorem*}[\ref{cohomtheorem}]
Let $D$ be a torus invariant Weil divisor on $X$. Then for
every torus invariant subscheme $V$ of $X$, every $i \geq 0$ and every $m \in M$:
\begin{equation*}
H^i_V\big(X, \mathcal{O}_X(D)\big)_m \cong H^{i - 1}(\hat{\Delta}_m, \hat{\Delta}_{V, m}; k).
\end{equation*}
\end{theorem*}

This theorem is an easy consequence of the standard characterization for the case
$V = X$ (\cite{Mustata3}), which says that $H^i\big(X, \sh{O}_X(D)\big)_m \cong
H^{i - 1}(\hat{\Delta}_m; k)$. We state it explicitly for reference purposes, as
it encompasses both, the case of global and local cohomology.

\paragraph{The circuit geometry of a toric variety.}

By above theorem, for an invariant divisor $D = \sum_{i = 1}^n c_i D_i$, the eigenspaces
$H^i_V\big(X, \sh{O}_X(D)\big)_m$ depend on the simplicial complexes $\hat{\Delta}$,
$\hat{\Delta}_V$ as well as on the position of the characters $m$ with respect to the
hyperplanes $H_i^{\underline{c}} = \{m \in M_\Q \mid l_i(m) = -c_i\}$, where $M_\Q =
M \otimes_\Z \Q$. The chamber decomposition of $M_\Q$ induced by the $H_i^{\underline{c}}$
(or their intersection poset) can be interpreted as the combinatorial type of $D$.
Our strategy will be to consider the variations of combinatorial types depending
on $\underline{c} = (c_1, \dots, c_n) \in \Q^n$. The solution to this discriminantal
problem is given by the {\em discriminantal arrangement} associated to the
vectors $l_1, \dots, l_n$, which has first been considered by Crapo \cite{Crapo84}
and Manin and Schechtman \cite{ManinSchechtman89}. The discriminantal arrangement
is constructed as follows. Consider the standard short exact sequence associated to $X$:
\begin{equation}\label{standardsequence}
0 \longrightarrow M_\Q \overset{L}{\longrightarrow} \Q^n \overset{G}{\longrightarrow}
A_\Q \longrightarrow 0,
\end{equation}
where $L$ is given by $L(m) = \big(l_1(m), \dots, l_n(m)\big)$,
and $A_\Q := A_{d - 1}(X) \otimes_\Z \Q$ is the rational divisor class group of $X$.
The matrix $G$ is called the {\em Gale} transform of $L$, and its $i$-th column
$D_i$ is the Gale transform of $l_i$. The most important property of the Gale transform
is that the linear dependencies among the $l_i$ and among the $D_i$ are inverted. That
is, for any subset among the $l_i$ which forms a basis, the complementary subset
of the $D_i$ forms a basis of $A_\Q$, and vice versa. Moreover, for every {\em circuit},
i.e. a minimal linearly dependent subset, $\circuit \subset \on$ the complementary set
$\{D_i \mid l_i \notin \circuit\}$ spans a hyperplane $H_\circuit$ in $A_\Q$. Then the
discriminantal arrangement is given by the hyperplane arrangement
\begin{equation*}
\{H_\circuit \mid \circuit \subset \on \text{ circuit}\}.
\end{equation*}
The stratification of $A_\Q$ by this arrangement then is in bijection with the
combinatorial types of the arrangements given by the $H_i^{\underline{c}}$ under
variation of $\underline{c}$.
As we will see, virtually all properties of $X$ concerning its birational geometry
and cohomology vanishing
of divisorial sheaves on $X$ depend on the discriminantal arrangement. In particular,
(see Proposition \ref{chamberprop}), the discriminantal arrangement coincides with
the hyperplane arrangement generated by the facets of the secondary fan. Ubiquitous
standard constructions such as the effective cone, nef cone, and
the Picard group can easily be identified as its substructures.

Another interesting aspect is that the discriminantal arrangement by itself
(or the associated matroid, respectively) represents a combinatorial invariant
of the variety $X$, which one can refer to as its {\em circuit geometry}.
This circuit geometry refines the combinatorial information coming with the
toric variety, that is, the fan $\Delta$ and the matroid structure underlying
the $l_i$ (i.e. their linear dependencies). It depends only on the $l_i$,
and even for two combinatorially equivalent
fans $\Delta$, $\Delta'$ such that corresponding sets of primitive vectors $l_1, \dots,
l_n$ and $l'_1, \dots, l'_n$ have the same underlying linear dependencies,
their associated circuit geometries are different in general.
This already is the case for surfaces, see, for
instance, Crapo's example of a plane tetrahedral line configuration (\cite{Crapo84},
\S 4). Falk (\cite{Falk94}, Example 3.2) gives a $3$-dimensional example.

\comment{
Another ingredient are arithmetic conditions: every circuit comes with a relation
\begin{equation}\label{circuitrelation}
\sum_{l_i \in \circuit} \alpha_i l_i = 0,
\end{equation}
where the $\alpha_i$ are unique up to multiplication by a common factor and can
always be
chosen to be integral. Relations of this type are standard in toric geometry, for
instance they play an important role in Reid's description of the toric minimal
model program \cite{Reid83} and the classification theory of smooth toric varieties
(see \cite{Oda}, \S 1.6, for instance).}

\comment{
As mentioned before, many geometric properties of $X$ are determined by
its circuit geometry.
For instance, an easy observation is the fact that the discriminantal arrangement
coincides
with the arrangement generated by the walls of the secondary fan (see Proposition
\ref{chamberprop}). Also note that the relations (\ref{circuitrelation}) give rise
to curve classes in $N_1(X)$. In particular,
the classes of invariant curves on $X$ are a subset of these classes.
These classes also share some properties with Batyrev's primitive collections
\cite{Batyrev91}; in particular if $X$ is projective (but not necessarily smooth
in our case), a subset of these classes
spans the extremal rays of the Mori cone (theorem \ref{moriconetheorem}).
}

\paragraph{Toric $1$-circuit varieties and the diophantine Frobenius problem.}

A special class of simplicial toric varieties, which we call {\em toric $1$-circuit
varieties}
are those with primitive vectors $l_1, \dots, l_{d + 1}$ forming a unique circuit. Such a
circuit comes with a relation
\begin{equation}\label{circuitrelation}
\sum_{i = 1}^{d + 1} \alpha_i l_i = 0
\end{equation}
where the $\alpha_i$ are nonzero integers whose largest common divisor is one. This relation
is unique up to sign and we assume for simplicity that $\alpha_i > 0$ for at least one $i$.
For a relation as in (\ref{circuitrelation}), we denote
$\mathbb{P}(\alpha_1, \dots, \alpha_{d + 1})$ the toric variety whose fan is generated by
maximal cones spanned by the sets $\{l_j \mid i \neq j\}$ for every $i$ with $\alpha_i > 0$.
Given that at least one $\alpha_i < 0$, we can likewise consider the variety
$\mathbb{P}(-\alpha_1, \dots, -\alpha_{d + 1})$ and it is not difficult to see that these
are the only two simplicial fans supported on the primitive vectors $l_1, \dots, l_{d + 1}$.
The integers $\alpha_1, \dots, \alpha_{d + 1}$ determine $\mathbb{P}(\alpha_1, \dots,
\alpha_{d + 1})$ uniquely up to a quotient by a finite group (which we suppress for this
exposition). In particular, if $\alpha_i > 0$ for all $i$, then the toric circuit variety is
isomorphic to the weighted projective space with weights $\alpha_i$.
If $\alpha_i < 0$ for at least one $i$, the associated toric $1$-circuit variety is a local
model for a flip (or flop)
$
\xymatrix{
\mathbb{P}(\alpha_1, \dots, \alpha_{d + 1}) \ar@{-->}[r] &
\mathbb{P}(-\alpha_1, \dots, -\alpha_{d + 1}).
}
$
This kind of operation shows up in the toric minimal model program and has been well-studied
(see \cite{Reid83}), and relations of type (\ref{circuitrelation}) play an important role for
the classification of toric varieties (see \cite{Oda}, \S 1.6, for instance).

The reason for studying toric $1$-circuit varieties in isolation is that they are the building
blocks for our arithmetic conditions on cohomology vanishing. Let $\mathbb{P}(\alpha_1, \dots,
\alpha_{d + 1})$ denote a weighted projective space with reduced weights $\alpha_i$ and $D$ the
unique $\Q$-effective generator of $A_{d - 1}\big(\mathbb{P}(\alpha_1, \dots, \alpha_{d + 1})\big)$.
Then by a standard construction we get
\begin{align*}
 & \dim H^0\big(\mathbb{P}(\alpha_1, \dots, \alpha_{d + 1}),
\sh{O}_{\mathbb{P}(\alpha_1, \dots, \alpha_{d + 1})}(nD)\big) \\
= &\Big\vert \{(k_1, \dots, k_{d + 1})
\in \N^{d + 1} \mid \sum_{i = 1}^{d + 1} k_i \alpha_i= n\} \Big\vert  =:
\operatorname{VP}_{\alpha_1, \dots, \alpha_{d + 1}}(n),
\end{align*}
for every $n \in \Z$,
where $\operatorname{VP}_{\alpha_1, \dots, \alpha_{d + 1}}$ is the so-called vector partition
function or denumerant function with respect to the $\alpha_i$. The problem of determining the
zero set of $\operatorname{VP}_{\alpha_1, \dots, \alpha_{d + 1}}$ (or the maximum of this set)
is quite famously known as the diophantine Frobenius problem. This problem is hard in general
(though not necessarily in cases of practical interest) and there does not exist a general closed
expression to determine the zero set.
For a survey of the diophantine Frobenius problem we refer to the book
\cite{RamirezAlfonsin}. For more general toric $1$-circuit varieties and higher cohomology
groups, the cohomology vanishing problem can be expressed in similar terms. Accepting the
fact that the general cohomology vanishing problem for toric varieties is at least as hard
as the diophantine Frobenius problem it still makes sense to simplify the general situation
by lifting vanishing conditions from the $1$-circuit case in a suitable way.

In subsection \ref{onecircuitvarieties} we will cover the cohomology vanishing problem
for toric $1$-circuit varieties somewhat more extensively than it would strictly be
necessary for proving our main theorems. The reason is that, from the general setup we
have to provide, we can derive essentially for free a complete characterization of
divisorial cohomology vanishing for these varieties. It should be instructive to
see this kind of results explicitly for the simplest possible cases. As the class of
toric $1$-circuit varieties contains the weighted projective spaces, our treatment
can be considered as toric supplement to standard references such as
\cite{Delorme}, \cite{Dolgachev82}, \cite{BeltramettiRobbiano} and should also
serve as a reference.

\paragraph{Lifting from the $1$-circuit case to the general case.}

The basic idea here is to transport the discriminantal arrangement from $A_\Q$ to some
diophantine analog in $A_{d - 1}(X)$. For any circuit $\circuit \subset \on$ there is
a short exact sequence
\begin{equation*}
0 \longrightarrow H_\circuit \longrightarrow A_\Q \longrightarrow A_{\circuit, \Q}
\longrightarrow 0.
\end{equation*}
By choosing one of the two possible simplicial fans supported on $\circuit$ (which needs not
necessarily be realized as a subfan
of $\Delta$), we have an induced orientation on $H_\circuit$ and we can identify $A_{\circuit, \Q}
= A_\circuit \otimes_\Z \Q \cong  \Q$ with the group of $\Q$-divisors on the corresponding
$1$-circuit variety. By lifting the surjection $A_\Q \rightarrow A_{\circuit, \Q}$ to its
integral counterpart $A_{d - 1}(X) \rightarrow A_\circuit$, we lift the zero set of the
corresponding vector partition function on $A_\circuit$ to $A_{d - 1}(X)$. By doing this
for every circuit $\circuit$, we construct in $A_{d - 1}(X)$ what we call the
{\em Frobenius discriminantal arrangement}. One can consider the Frobenius
discriminantal arrangement as an
arithmetic thickening of the discriminantal arrangement. This thickening in general
is just enough to enlarge the relevant strata in the discriminantal arrangement such
that it encompasses the Kawamata-Viehweg-like theorems. To derive other vanishing
results, our analysis will mostly be concerned with analyzing the birational geometry
of $X$ and its implications on the combinatorics of the discriminantal arrangement, and
the transport of this analysis to the Frobenius arrangement.

\paragraph{Overview.}

Section \ref{weilcohomology} is devoted to the proof of theorem \ref{cohomtheorem}.
In section \ref{onecircuitvarieties} we give a complete characterization of
cohomology vanishing for toric $1$-circuit varieties.
In section \ref{discriminantalarrangementsandsecfans} we survey discriminantal arrangements, secondary fans, and rational aspects of cohomology vanishing.
Several technical facts will
be collected which are important for the subsequent sections. Section
\ref{arithmeticaspects} contains all the essential results of this work.
In \ref{nonstandardvanishing} we will prove our main arithmetic vanishing results.
These will be applied in \ref{nonstandardsurfacevanishing} to give a quite
complete characterization of cohomology vanishing for toric surfaces.
Section \ref{mcmsection} is devoted to maximally Cohen-Macaulay modules.

\paragraph{Acknowledgments.}

Thanks to Laurent Bonavero, Michel Brion, Lutz Hille, Vic Reiner, and
Jan Stienstra for discussion and useful hints.

\section{Toric Preliminaries}
\label{prelim}\label{weilcohomology}

In this section we first introduce notions from toric geometry which will be used
throughout the rest of the paper. As general reference for toric varieties we use
\cite{Oda}, \cite{Fulton}. We will always work over an algebraically closed
ground field $k$.

Let $\Delta$ be a fan
in the rational vector space $N_\Q := N \otimes_\Z \Q$ over a lattice
$N \cong \mathbb{Z}^d$.
Let $M$ be the lattice dual to $N$, then the elements of $N$ represent linear forms
on $M$ and we write $n(m)$ for the canonical pairing $N \times M \rightarrow \mathbb{Z}$,
where $n \in N$ and $m \in M$. This pairing extends naturally over $\Q$, $M_\Q
\times N_\Q \rightarrow \Q$. Elements of $M$ are denoted by $m$, $m'$,
etc. if written additively, and by
$\chi(m)$, $\chi(m')$, etc. if written multiplicatively, i.e. $\chi(m + m') =
\chi(m)\chi(m')$. The lattice $M$ is identified with the group of characters of the
algebraic torus $T = \operatorname{Hom}(M,k^*) \cong (k^*)^d$ which acts on the toric
variety $X = X_\Delta$ associated to $\Delta$.
Moreover, we will use the following notation:

\begin{itemize}
\item cones in $\Delta$ are denoted by small greek letters $\rho, \sigma, \tau,
\dots$, their natural partial order by $\prec$, i.e. $\rho \prec \tau$ iff $\rho
\subseteq \tau$;
\item $\vert \Delta \vert := \bigcup_{\sigma \in \Delta} \sigma$ denotes the
support of $\Delta$;
\item for $0 \leq i \leq d$ we denote $\Delta(i) \subset \Delta$ the set of
$i$-dimensional cones; for $\sigma \in \Delta$, we denote $\sigma(i)$ the set of
$i$-dimensional faces of $\sigma$;
\item $U_\sigma$ denotes the associated affine toric variety for any $\sigma \in
\Delta$;
\item $\check{\sigma} := \{m \in M_\mathbb{Q} \mid n(m) \geq 0
\text{ for all $n \in \sigma$}\}$ is the cone {\it dual} to $\sigma$;
\item $\sigma^\bot = \{m \in M_\mathbb{Q} \mid n(m) = 0 \text{ for
all } n \in \sigma \}$;
\item $\sigma_M := \check{\sigma} \cap M$ is the submonoid of $M$ associated 
to $\sigma$.
\end{itemize}

We will mostly be interested in the structure of $\Delta$ as a combinatorial
cellular complex. For this, we make a few convenient identifications. We always denote
$n$ the cardinality of $\Delta(1)$. i.e. the number of $1$-dimensional cones ({\em rays})
and $\on := \{1, \dots, n\}$.
The primitive vectors along rays are denoted $l_1, \dots, l_n$, and, by abuse of notion, we
will usually
identify the sets $\Delta(1)$, the set of primitive vectors, and $\on$.
Also, we will often identify $\sigma \in \Delta$ with the set $\sigma(1) \subset \on$.
With these identifications, and using the natural order of $\on$, we obtain a combinatorial
cellular complex with support $\on$; we may consider this complex as a combinatorial
model for $\Delta$. In the case where $\Delta$ is simplicial, this complex is just a
combinatorial simplicial complex in the usual sense. If $\Delta$ is not simplicial,
we consider the {\em simplicial cover} $\hat{\Delta}$ of $\Delta$, modelled on $\on$:
some subset $I
\subset \on$ is in $\hat{\Delta}$ iff there exists some $\sigma \in \Delta$ such that
$I \subset \sigma(1)$.
The identity on $\on$ then induces a surjective morphism $\hat{\Delta} \longrightarrow
\Delta$ of combinatorial cellular complexes. This morphism has a natural representation
in terms of fans. We can identify $\hat{\Delta}$ with the fan in $\Q^n$ which is defined
as follows. Let $e_1, \dots, e_n$ be the standard basis of $\Q^n$, then for any set
$I \subset \on$, the vectors $\{e_i\}_{i \in I}$ span a cone over $\Q_{\geq 0}$
iff there exists $\sigma \in \Delta$ with $I \subset \sigma(1)$. The associated
toric variety
$\hat{X}$ is open in $\mathbb{A}^n_k$, and the vector space homomorphism defined by mapping
$e_i \mapsto l_i$ for $i \in \on$ induces a map of fans $\hat{\Delta} \rightarrow \Delta$.
The induced morphism $\hat{X} \rightarrow X$ is the quotient presentation due to
Cox \cite{Cox}. We will not make explicit use of this construction, but it may be useful
to have it in mind.

An important fact used throughout this work is the following exact sequence which exists
for any toric variety $X$ with associated fan $\Delta$:
\begin{equation}\label{standardZsequence}
M \overset{L}{\longrightarrow} \Z^n \longrightarrow A_{d - 1}(X)
\longrightarrow 0.
\end{equation}
Here $L(m) = (l_1(m), \dots, l_n(m))$, i.e. as a matrix, the primitive
vectors $l_i$ represent the row vectors of $L$. Note that $L$ is injective
iff $\Delta$ is not contained in a proper subspace of $N_\Q$.
The sequence follows from the fact that every
Weil divisor $D$ on $X$ is rationally equivalent to a $T$-invariant Weil divisor, i.e.
$D \sim \sum_{i = 1}^n c_i D_i$, where $\underline{c} = (c_1, \dots, c_n) \in \Z^n$
and $D_1, \dots, D_n$, the $T$-invariant irreducible divisors of $X$. Moreover, any
two
$T$-invariant divisors $D$, $D'$ are rationally equivalent if and only if there exists
$m \in M$ such that $D - D' = \sum_{i = 1}^n l_i(m) D_i$. To every Weil divisor $D$,
one associates its divisorial sheaf $\sh{O}_X(D) = \sh{O}(D)$ (we will omit the
subscript $X$ whenever there is no ambiguity), which is a reflexive sheaf of rank one and
locally free if and only if $D$ is Cartier. Rational equivalence classes
of Weil divisors are in bijection with isomorphism
classes of divisorial sheaves. If $D$ is $T$-invariant, the sheaf $\sh{O}(D)$ acquires
a $T$-equivariant structure and the equivariant isomorphism classes of sheaves
$\sh{O}(D)$ are one-to-one with $\Z^n$.

Now Consider a closed $T$-invariant subscheme $V \subseteq X$. Then for any $T$-invariant
Weil divisor $D$ there are induced linear representations of $T$ on the local cohomology
groups $H^i_V\big(X, \sh{O}(D)\big)$. In particular, each such module has a natural
eigenspace decomposition
\begin{equation*}
H^i_V\big(X, \sh{O}(D)\big) \cong \bigoplus_{m \in M}H^i_V\big(X, \sh{O}(D)\big)_m.
\end{equation*}
The eigenspaces $H^i_V\big(X, \sh{O}(D)\big)_m$ can be characterized by the relative
cohomologies of certain simplicial complexes.
For any $I \subset \on$ we denote $\hat{\Delta}_I$ the maximal subcomplex of $\hat{\Delta}$
which is supported on $I$. We denote $\hat{\Delta}_V$ the simplicial cover of the
fan associated to the complement of the reduced subscheme underlying $V$ in $X$.
Correspondingly, for $I \subset \on$ we denote $\hat{\Delta}_{V, I}$ the maximal
subcomplex of $\hat{\Delta}_V$ which is supported on $I$. If $\uc \in \Z^n$ is fixed,
and $D = \sum_{i \in \on} c_i D_i$, then every $m \in M$ determines a subset
$I(m)$ of $\on$ which is given by
\begin{equation*}
I(m) = \{i \in \on \mid l_i(m) < -c_i\}.
\end{equation*}
Then we will write $\hat{\Delta}_m$ and $\hat{\Delta}_{V, m}$ instead of
$\hat{\Delta}_{I(m)}$ and $\hat{\Delta}_{V, I(m)}$, respectively. In the case
where $\Delta$ is generated by just one cone $\sigma$, we will also write
$\hat{\sigma}_m$, etc.
With respect to these notions we get:

\begin{theorem}\label{cohomtheorem}
Let $D \in \mathbb{Z}^{\Delta(1)}$ be a $T$-invariant Weil divisor on $X$. Then for
every $T$-invariant closed subscheme $V$ of $X$, every $i \geq 0$ and every $m \in M$ there
exists an isomorphism of $k$-vector spaces
\begin{equation*}
H^i_V\big(X, \mathcal{O}(D)\big)_m \cong H^{i - 1}(\hat{\Delta}_m, \hat{\Delta}_{V, m}; k).
\end{equation*}
\end{theorem}

Note that here $H^{i - 1}(\hat{\Delta}_m, \hat{\Delta}_{V, m})$ denotes the {\em reduced}
relative cohomology group of the pair $(\hat{\Delta}_m, \hat{\Delta}_{V, m})$.

\begin{proof}
For $V = X$ it follows from \cite{Mustata3}, \S 2 that $H^i\big(X, \mathcal{O}(D)\big)_m \cong
H^{i - 1}(\hat{\Delta}_m; k)$ and $H^i\big(X \setminus V, \mathcal{O}(D)\big)_m \cong
H^{i - 1}(\hat{\Delta}_{V, m}; k)$. Then the assertion follows from comparing the long exact
relative cohomology sequence of the pair $(\hat{\Delta}_m, \hat{\Delta}_{V, m})$ with the
long exact local cohomology sequence with respect to $X$ and $V$ in degree $m$.
\end{proof}

We mention a special case of this theorem, which follows from the long exact
cohomology sequence.

\begin{corollary}\label{localcohomcorollary}
Let $X = U_\sigma$ and $V$ a $T$-invariant closed subvariety of $X$ and denote $\hat{\sigma}$ the
simplicial model for the fan generated by $\sigma$. Then for every
$m \in M$ and every $i \in \Z$:
\begin{equation*}
H^i_V\big(X, \sh{O}(D)\big)_m =
\begin{cases}
0 & \text{ if }\, \hat{\sigma}_m = \emptyset, \\
H^{i - 2}(\hat{\sigma}_{V, m}; k) & \text{ else}.
\end{cases}
\end{equation*}
\end{corollary}

\section{Toric $1$-Circuit Varieties}\label{onecircuitvarieties}

We now study divisorial cohomology vanishing for the simplest possible toric varieties which
are not affine. Consider primitive vectors $l_1, \dots, l_n$, which form a so-called
{\em  circuit}, i.e. a minimally linearly dependent set in $N$. Then there exists a
relation
\begin{equation*}
\sum_{i \in \on} \alpha_i l_i = 0,
\end{equation*}
which is unique up to a common multiple of the $\alpha_i$, and the $\alpha_i$ are nonzero.
For simplifying the discussion, we will assume that the $l_i$ generate a submodule $N_\on$ of finite
index in $N$, in particular, we have $n = d + 1$.
Without loss of generality, we will assume that the $\alpha_i$ are integral and $\gcd\{\vert\alpha_i\vert\}_{i \in \on} = 1$.
For a fixed choice of the $\alpha_i$, we
have a partition $\on = \ocircuit^+ \coprod \ocircuit^-$, where $\ocircuit^\pm = \{i \in \on \mid \pm \alpha_i > 0\}$.
This decomposition depends only on the signs of the $\alpha_i$; flipping the signs exchanges
$\ocircuit^+$ and $\ocircuit^-$. We want to keep track of this two possibilities and call
the choice of $\ocircuit^+ \coprod \ocircuit^- =: \ocircuit$ the {\em oriented circuit} with
underlying circuit $\on$, and $-\ocircuit := -\ocircuit^+ \coprod -\ocircuit^-$ its inverse,
where $-\ocircuit^\pm := \ocircuit^\mp$. The primitive vectors $l_i$ can support at most
two simplicial fans, each corresponding to an oriented circuit. For fixed orientation $\ocircuit$,
we denote $\Delta = \Delta_\ocircuit$ the fan whose maximal cones are generated by
$\on \setminus \{i\}$, where $i$ runs over the elements of $\ocircuit^+$. The only
exception for this procedure is the case where $\ocircuit^+$ is empty, which we leave
undefined. The associated toric variety $X_{\Delta_\ocircuit}$ is simplicial and
quasi-projective.

\begin{definition}
We call a toric variety $X = X_{\Delta_\ocircuit}$ associated to an oriented circuit a {\em toric
$1$-circuit variety}.
\end{definition}

Now let us  assume that the sublattice $N_\on$ which is generated by the $l_i$ is saturated, i.e.
$N_\on \otimes_\Z \Q \cap N = N_\on$. Then we have an exact sequence
\begin{equation}\label{circuitexact}
M \overset{L}{\longrightarrow} \Z^n \overset{G}{\longrightarrow} A \longrightarrow 0,
\end{equation}
where $L = (l_1, \dots, l_n)$ considered as a tuple of linear forms on $M$, $A \cong \Z$ and
$G =  (\alpha_1, \dots, \alpha_n)$ a $(1 \times n)$-matrix, i.e. we can consider the
$\alpha_i$ as the {\em Gale transform} of the $l_i$. Conversely, if the $\alpha_i$ are given,
then the $l_i$ are determined up to a $\Z$-linear automorphism of $M$.
We will make more extensively use of the Gale transform later on. For generalities
we refer to \cite{OdaPark} and \cite{GKZ}. 

In the case that $N_\on$ is not saturated, we can formally consider the inclusion of
$N_\on$ as the image of a saturated sublattice of an injective endomorphism $\xi$ of $N$.
The inverse images of the $l_i$ with respect to $\xi$ satisfy the same relation as the
$l_i$. Therefore, a general toric circuit variety is completely specified by $\xi$ and the
integers $\alpha_i$. More precisely, a toric $1$-circuit variety is specified by the Gale
duals $l_i$ of the $\alpha_i$ and a an injective endomorphism $\xi$ of $N$ with the
property that $\xi(l_i)$ is primitive in $N$ for every $i \in \on$.

\begin{definition}
Let $\underline{\alpha} = (\alpha_1, \dots, \alpha_n) \in \Z^n$ with $\alpha_i \neq 0$ for
every $i$ and $\gcd\{\vert\alpha_i\vert\}_{i \in \on} = 1$,
$\ocircuit$ the unique oriented circuit with $\ocircuit^+ = \{i \mid \alpha_i > 0\}$,
and $\xi : N \longrightarrow N$ an injective endomorphism of $N$ which maps
the Gale duals of the $\alpha_i$ to primitive elements $l_i$ in $N$. Then we denote
$\mathbb{P}( \underline{\alpha}, \xi)$ the toric $1$-circuit variety
associated to the fan $\Delta_\ocircuit$ spanned by the primitive vectors
$\xi(l_i)$.
\end{definition}

The endomorphism $\xi$ translates into an isomorphism
\begin{equation*}
\mathbb{P}(\underline{\alpha}, \xi) \cong \mathbb{P}(\underline{\alpha}, \operatorname{id}_N) / H,
\end{equation*}
where $H \cong \spec{k[N / N_\on]}$. Note that in positive characteristic, $H$ in general is a
group scheme rather than a proper algebraic group.

In sequence (\ref{circuitexact}), we can
identify $A$ with the divisor class group $A_{d - 1}\big(\mathbb{P}(\underline{\alpha},
\operatorname{id}_N )\big)$. Similarly, we get  $A_{d - 1}\big(\mathbb{P}(\underline{\alpha},
\xi )\big) \cong A \oplus H$ and the natural surjection from
$A_{d - 1}\big(\mathbb{P}(\underline{\alpha}, \xi )\big)$ onto
$A_{d - 1}\big(\mathbb{P}(\underline{\alpha}, \operatorname{id}_N )\big)$
just projects away the torsion part.

The $\mathbb{P}(\underline{\alpha}, \xi)$ are an important building block for general
toric varieties and therefore they will play a distinguished role in later sections.
In fact, to every
extremal curve $V(\tau)$ in some simplicial toric variety $X$, there is associated some variety
$\mathbb{P}(\underline{\alpha}, \xi)$ whose fan $\Delta_\ocircuit$ is a subfan of $\Delta$
and $\mathbb{P}(\underline{\alpha}, \xi)$ and which embeds as an open invariant subvariety of $X$.
If $\vert \ocircuit^+ \vert \notin \{n, n - 1\}$, the primitive vectors $l_i$ span a
convex polyhedral cone, giving rise to an affine toric variety $Y$ and a canonical
morphism $\pi: \mathbb{P}(\underline{\alpha}, \xi) \longrightarrow Y$ which is a partial
resolution of singularities. Sign change
$\underline{\alpha} \rightarrow -\underline{\alpha}$ then encodes the transition from
$\ocircuit$ to $-\ocircuit$ and provides a local model for well-known
combinatorial operation which called {\em bistellar operation} \cite{Reiner99} 
or {\em modification of a triangulation} \cite{GKZ}. In toric geometry usually it is
also called a {\em flip}:
\begin{equation*}
\xymatrix{
\mathbb{P}(\underline{\alpha}, \xi) \ar[rd]^\pi \ar@{-->}[rr]^{\text{flip}} & &
\mathbb{P}(-\underline{\alpha}, \xi) \ar[ld]_{\pi'} \\
& Y &
}
\end{equation*}
For $\vert \ocircuit^+ \vert = d - 1$, we identify $\mathbb{P}(-\underline{\alpha}, \xi)$
with $Y$ and just obtain a blow-down.

In the case $\alpha_i > 0$ for all $i$ and $\xi = \operatorname{id}_N$, we just recover
the usual weighted projective spaces. In many respects, the spaces
$\mathbb{P}(\underline{\alpha}, \xi)$ can be treated the same way as
has been done in the standard references for weighted projective spaces, see
\cite{Delorme}, \cite{Dolgachev82}, \cite{BeltramettiRobbiano}.
In our setting there is the slight simplification that we naturally can  assume that
$\gcd\{\vert\alpha_j\vert\}_{j \neq i} = 1$ for every $i \in \on$, which eliminates
the need to discuss reduced weights.

\begin{remark}
On  $\mathbb{P}(\underline{\alpha}, \xi)$ the sheaves $\sh{O}(D)$ for
$D \in A_{d - 1}\big( \mathbb{P}(\underline{\alpha}, \xi) \big) $ can be constructed
via the homogeneous coordinate ring. For sake of information we present the relevant data
without proof and refer to \cite{Cox} for details (see also \cite{perling1}).
The homogeneous coordinate ring is given by the polynomial ring $S := k[x_1, \dots,
x_{d + 1}]$ with grading $S = \bigoplus_{\alpha \in A} S_\alpha$, where $A :=
A_{d - 1}\big( \mathbb{P}(\underline{\alpha}, \xi) \big)$ and
$\deg_A x_i = [D_i] \in A$. The grading is induced by the action
of the group scheme $\tilde{A} = \spec k[A]$ on $\spec S =
\mathbb{A}^{d + 1}_k$. The irrelevant ideal $B \subset S$ is of the form $B =
\langle x_i \mid i \in \ocircuit^+ \rangle$ and the variety $\mathbb{P}(\underline{\alpha},
\xi)$ then is a good quotient of $\mathbb{A}^{d + 1}_k \setminus V(B)$ by
$\tilde{A}$. By sheafification, every divisorial sheaf is of the form
$\widetilde{S(\alpha)}$, for $\alpha \in A$.
If $A$ is torsion free, then $\tilde{A} \cong
k^*$ and the $\Z$-grading is given by $\deg_\Z x_i = \alpha_i$.
\end{remark}

\subsection{Singularities and Picard group}

In general, $\mathbb{P}(\underline{\alpha}, \xi)$ is not smooth and its singularities
depend on the degree of the sublattices of $N$ spanned by subsets of the $l_i$ with
respect to their saturations.

\begin{definition}
Let $l_1, \dots, l_n \in N$ be primitive vectors and let $I \subset \on$. Then we denote $N_I$
the submodule of $N$ spanned by the $l_i$ with $i \in I$,  $\bar{N}_I$ its saturation in $N$, and
$r_I$ the index of $N_I$ in $\bar{N}_I$.
\end{definition}

Up to the global torsion relative to $\xi$, the structure of its singularities is encoded in
$\underline{\alpha}$:

\begin{lemma}\label{cycliclemma}
Assume that $\xi$ is an automorphism of $N$, then $\bar{N}_I / N_I$ is cyclic and for every
proper subset $I \subset \on$, we have $r_I =
\gcd\{\vert \alpha_i \vert\}_{i \in \on \setminus I}$. Moreover, $r_{\{i\}} = 1$
for every $i \in \on$.
\end{lemma}

\begin{proof}
As $\gcd\{\vert\alpha_i\vert\}_{i \in \circuit} = 1$, the relation $\sum_{i \in \circuit}
\alpha_i l_i = 0$ is unique up to sign and $\sum_{i \in \circuit \setminus I} \alpha_i l_i$
$=: l_I \in N_I$. Let $\lambda := \gcd\{\vert\alpha_i\vert\}_{i \in \circuit \setminus I}$
and denote $l'_I \in \bar{N}_I$ with $l_I = \lambda \cdot l'_I$. Clearly, if $\lambda \neq
1$, then $l'_I$ is not contained in $N_I$.
As $L$ generates $N$, the submodule $\bar{N}_I$  is spanned by $N_I$ and $l'_I$. Thus
$\bar{N}_I / N_I$ is cyclic and generated by the image
$\bar{l}_I$ of $l'_I$ in $\bar{N}_I / N_I$ and $\lambda$ must be a multiple of the order
of $\bar{l_I}$. But $\lambda$ is the least multiple such that $\lambda l'_I \in N_I$,
as otherwise there would exist an integral relation $\sum_{i \in \circuit} \beta_i l_i = 0$
with $\gcd\{\vert\beta_i\vert\}_{i \in \circuit} = 1$ different from the original one,
contradicting its uniqueness. The last assertion follows from the fact that the $l_i$ are
primitive.
\end{proof}

Recall that the maximal cones of $\Delta_\ocircuit$ are spanned by the complementary
rays of $i$ for every $i \in \ocircuit^+$. Therefore:

\begin{corollary}\label{cycliccorollary}
The variety $\mathbb{P}(\underline{\alpha}, \operatorname{id}_N)$ has cyclic quotient
singularities of degree $\alpha_i$ for every $i \in \ocircuit^+$.
\end{corollary}

In presence of a nontrivial $\xi$, there are some more factors to take into account:

\begin{definition}
Let $I \subset \ocircuit$, then we denote $s_I := r_I^{-1}\vert \bar{N}_I / N_I \vert$.
If $I = \sigma(1)$ for some
$\sigma \in \Delta_\ocircuit$, we write $s_I =: s_\sigma$, and for $I = \circuit$
we simply write $s_\on =: s$.
\end{definition}

Note that $s = \vert \det \xi \vert$. The group $\pic\big(\mathbb{P}(\underline{\alpha}, \xi)\big)$
embeds into $A_{d - 1}(\big(\mathbb{P}(\underline{\alpha}, \xi )\big)$ as a subgroup of finite
index. Therefore, if $\xi$ is an automorphism, it follows that $\pic\big(\mathbb{P}(\underline{\alpha}, \xi)\big)
\cong \Z$. In general, by the isomorphism $\mathbb{P}(\underline{\alpha}, \xi) \cong \mathbb{P}(\underline{\alpha},
\operatorname{id}_N) / H$, the group $\pic\big(\mathbb{P}(\underline{\alpha}, \xi)\big)$ embeds into
$\pic\big(\mathbb{P}(\underline{\alpha}, \operatorname{id}_n)\big)$ as a subgroup of index $s$ via pull-back.
For simplicity, it is suitable
to identify $\pic\big(\mathbb{P}(\underline{\alpha}, \xi)\big)$ with its image in 
$\pic\big(\mathbb{P}(\underline{\alpha}, \operatorname{id}_N)\big)$. We have:

\begin{proposition}\label{circuitpic}
$\pic\big(\mathbb{P}(\underline{\alpha}, \xi)\big)$ is generated by
$s \cdot \lcm\{\alpha_i\}_{i \in \ocircuit^+}$.
\end{proposition}

\begin{proof}
It suffices to prove that $\pic\big(\mathbb{P}(\underline{\alpha}, \operatorname{id}_N)\big)$ is
generated by $\lcm\{\alpha_i\}_{i \in \ocircuit^+}$.
Let $D = \sum_{i \in \circuit} c_i D_i$ be the generator of
$\pic\big(\mathbb{P}(\underline{\alpha}, \operatorname{id}_N)\big)$. Then 
$D$ is specified by a collection $\{m_i \in M \mid i \in \ocircuit^+\}$ such that $l_i(m_j) =
c_i$ for all $j$ and all $j \neq i \in \circuit$. As changing the $m_i$ to $m_i + m$ for some $m \in M$
just changes the linearization of $\sh{O}(D)$, but not the linear equivalence class of $D$, we can
always assume that one of the $m_i$ is zero. This implies that $c_j = 0$ for all $j \neq i$, and
$l_i(m_j) = c_i $ for all $j \neq i$ and $l_k(m_j) = 0$ for all $k \neq i, j$.
Using the relation $\sum_{k \in \circuit}
\alpha_k l_k = 0$, we thus obtain $c_i = -l_i(m_j) = \frac{\alpha_j}{\alpha_i} l_j (m_j)$
for every $j \in \ocircuit^+$. Therefore $c_i D_i$ is a multiple of $\alpha_i^{-1} \cdot
\lcm\{\alpha_j\}_{j \in \ocircuit^+} \cdot D_i$ and thus, by identifying $\alpha_i$ with $D_i$
in $\pic\big(\mathbb{P}(\underline{\alpha}, \operatorname{id}_N)\big)$, a multiple of
$\lcm\{\alpha_j\}_{j \in \ocircuit^+}$. On the other hand, it is easy to see that every such multiple
yields a Cartier divisor on $\mathbb{P}(\underline{\alpha}, \operatorname{id}_N)$.
\end{proof}

\begin{remark}
It follows that a Cartier divisor has no torsion part in $A_{d - 1}\big(\mathbb{P}(\underline{\alpha}, \xi )\big)$,
i.e. the embedding
$\pic\big(\mathbb{P}(\underline{\alpha}, \xi)\big) \hookrightarrow A_{d - 1}
\big(\mathbb{P}(\underline{\alpha}, \xi )\big)$ factorizes
through the section $A_{d - 1}\big(\mathbb{P}(\underline{\alpha}, \xi )\big) / H
\rightarrow A_{d - 1}\big(\mathbb{P}(\underline{\alpha}, \xi )\big)  \cong \Z \oplus H$.
\end{remark}

As a consequence we have:

\begin{corollary}\label{primcorollary}
Let $i \in \ocircuit^+$, then the invariant prime divisor $D_i$ on
$\mathbb{P}(\underline{\alpha}, \xi)$ is Cartier if and only if $s = 1$ and
$\alpha_j = 1$ for every $j \in \ocircuit^+ \setminus \{ i \}$. In particular,
$\mathbb{P}(\underline{\alpha}, \xi)$ is smooth if and only if
$D_i$ is Cartier for every $i \in \ocircuit^+$.
\end{corollary}

\begin{proof}
Let first $s = 1$ and $\alpha_j = 1$ for every $j \in \ocircuit^+ \setminus \{ i \}$.
Then by Corollary \ref{cycliccorollary}, the union of the $U_{\sigma_j}$, where
$j \in \ocircuit^+ \setminus \{ i \}$, is smooth and $D_i$ restricted to this union
is Cartier. As $D_i$ is trivial on $U_{\sigma_i}$, it extends as Cartier divisor with
trivial restriction on $U_{\sigma_i}$.
Now let $D_i$ Cartier, then from proposition \ref{circuitpic} we conclude
$D_i = (\frac{s}{\alpha_i} \lcm\{\alpha_j\}_{j \in \ocircuit^+}) \cdot D_i$,
and thus $s = 1$ and $\alpha_j = 1$ for all $i \neq j \in \ocircuit^+$.
\end{proof}

\subsection{General cohomology vanishing.}
\label{onecircuitgeneralcohomologyvanishing}

In light of Theorem \ref{cohomtheorem}, for cohomology vanishing on a toric $1$-circuit
variety, we have to consider the reduced cohomology of simplicial complexes associated
to its fan:

\begin{lemma}\label{circtoplemma}
Let $I \subset \on$, such that $I \neq \ocircuit^+$, then
$H^i((\hat{\Delta}_\ocircuit)_I; k) = 0$ for all $i$. Moreover,
\begin{equation*}
(\hat{\Delta}_\ocircuit)_{\ocircuit^+} \cong S^{\vert\ocircuit^+\vert - 2} \quad \text{ and }
\quad (\hat{\Delta}_\ocircuit)_{\ocircuit^-} \cong B^{\vert\ocircuit^-\vert - 1},
\end{equation*}
where $B^k$ is the $k$-ball, with $B^{-1} := \emptyset$.
\end{lemma}

\begin{proof}
It is easy to see that $(\hat{\Delta}_\ocircuit)_{\ocircuit^+}$ corresponds to the boundary
of the $(\vert\ocircuit^+\vert - 1)$-simplex, so it is homeomorphic to
$S^{\vert\ocircuit^+\vert - 2}$. Similarly, $\{l_i\}_{i \in \ocircuit^+}$ span a simplicial
cone in $\Delta_\ocircuit$ and thus $(\hat{\Delta}_\ocircuit)_{\ocircuit^-} \cong
B^{\vert\ocircuit^-\vert - 1}$.
Now assume there exists $i \in \ocircuit^+ \setminus I$, then $I$ is a face of the cone
$\sigma_i$ and $(\hat{\Delta}_\ocircuit)_I$ is contractible. On the other hand, if
$\ocircuit^+$ is a proper subset of $I$, the set $I \cap \ocircuit^-$ spans a cone $\tau$
in $\Delta_\ocircuit$. The simplicial complex $\hat{\Delta}_I$ then is homeomorphic to
a simplicial decomposition of the $(\vert \ocircuit^+\vert - 1)$-ball with center $\tau$
and boundary $(\hat{\Delta}_\ocircuit)_{\ocircuit^+}$.
\end{proof}

By sequence (\ref{standardZsequence}), the cohomology of a divisor $D$ depends only on the choice
of a $T$-invariant representative $D = \sum_{i \in \on} c_i D_i$ with $c_i \in \Z$. This choice is
unique up to a twist by a character in $M$, i.e. any divisor of the form $\sum_{i \in \on}
l_i(m) D_i$ for some $m \in M$ is trivial.
This can be interpreted more geometrically in terms of hyperplane
arrangements in $M_\Q$. For a given choice of $\underline{c} = (c_1, \dots, c_n) \in \Z^n$ we set
\begin{equation*}
H_i^{\underline{c}} := \{m \in M_\Q \mid l_i(m) = -c_i\}.
\end{equation*}
Then, replacing $c_i$ by $c_i + l_i(m)$ for some $m \in M$ then corresponds to an integral
translation of the hyperplane arrangement $\{H_i^{\underline{c}}\}_{i \in \on}$ by $-m$.
This hyperplane arrangement induces a chamber decomposition of $M_\Q$. If $D \sim 0$,
then the maximal chambers are all unbounded. If $D \nsim 0$, then we get precisely one additional
chamber which is bounded (recall that in this section $n = d + 1$ and the $l_i$ form a circuit).
For this chamber, there are two possibilities. Either it is given by points $m \in M_\Q$ such
that $l_i(m) \leq -c_i$ for $i \in \ocircuit^-$ and  $l_i(m) \geq -c_i$ for $i \in \ocircuit^-$, or
vice versa. Let us say in the first case that this chamber has {\em signature} $\ocircuit^+$and
in the second case it has signature $\ocircuit^-$ (we will define signatures more generally in
section \ref{circuitsdiscriminantal}).

To determine cohomology vanishing we have to determine its signature and whether it
contains lattice points. As already explained in the introduction, the number of lattice points
in a bounded chamber is given by the vector partition function. Similarly, the number of
lattice points $m$ such that $l_i(m) \geq -c_i$ for $i \in \ocircuit^+$ and $l_i(m) < -c_i$
for $i \in \ocircuit^-$ coincides with the cardinality of the following set:
\begin{equation*}
\{k_1, \dots, k_{d + 1} \in \N^{d + 1} \mid k_i > 0 \text{ for } i \in
\ocircuit^- \text{ and }\sum_{i \in \ocircuit^+} k_i D_i - \sum_{i \in \ocircuit^-} k_i D_i = D\}.
\end{equation*}
The set of rational divisor classes in $A_{d - 1}\big(\mathbb{P}(\underline{\alpha}, \xi)\big)_\Q
\cong \Q$ corresponding to torus
invariant divisors whose associated bounded chamber has signature either $\ocircuit^+$
or $\ocircuit^-$ corresponds precisely to the two open intervals $(-\infty, 0)$ and
$(0, \infty)$, respectively, in  $A_{d - 1}\big(\mathbb{P}(\underline{\alpha}, \xi)\big)_\Q$.
In the integral case, we can consider {\em arithmetic thickenings} of these intervals as follows:

\begin{definition}\label{fcdef}
We denote $F_\ocircuit \subset A_{d - 1}\big(\mathbb{P}(\underline{\alpha}, \xi)\big)$
the complement of the semigroup of the form
$\sum_{i \in \ocircuit^-} c_i  D_i - \sum_{i \in \ocircuit^+} c_i D_i$, where $c_i \in \N$ for all $i$
with $c_i > 0$ for $i \in \ocircuit^+$.
\end{definition}

The set $F_\ocircuit$ is the complement of the set of classes whose associated chamber
has signature $\ocircuit^-$ and contains a lattice point.
With this we can give a complete characterization of global cohomology vanishing:

\begin{proposition}\label{circuitglobalcohom}
Let $\mathbb{P}(\underline{\alpha}, \xi)$ be as before with associated fan $\Delta_\ocircuit$
and $D \in A_{d - 1}\big(\mathbb{P}(\underline{\alpha}, \xi)\big)$, then:
\begin{enumerate}[(i)]
\item $H^i\big(\mathbb{P}(\underline{\alpha}, \xi), \sh{O}(D)\big) = 0$ for
$i \neq 0, \vert \ocircuit^+ \vert - 1$;
\item $H^{\vert \ocircuit^+ \vert - 1}\big(\mathbb{P}(\underline{\alpha}, \xi),
\sh{O}(D)\big) = 0$ iff $D \in F_\ocircuit$;
\item if $\ocircuit^+ \neq \circuit$, then $H^0\big(\mathbb{P}(\underline{\alpha}, \xi),
\sh{O}(D)\big) \neq 0$;
\item if $\ocircuit^+ = \circuit$, then $H^0\big(\mathbb{P}(\underline{\alpha}, \xi),
\sh{O}(D)\big) = 0$ iff $D \in F_{-\ocircuit}$.
\end{enumerate}
\end{proposition}

\begin{proof}
The proof is immediate. Just observe that the simplicial complex $(\hat{\Delta}_\ocircuit)_m$,
for $m$ an element in the bounded chamber, coincides either with $(\hat{\Delta}_\ocircuit)_{\ocircuit^+}$
or $(\hat{\Delta}_\ocircuit)_{\ocircuit^-}$.
\end{proof}

Another case of interest is where $\ocircuit^+ \neq \circuit$ and $V = V(\tau)$, where
$\tau$ is the cone spanned by the
$l_i$ with $i \in \ocircuit^-$, i.e. $V$ is the unique maximal complete torus invariant subvariety
of $\mathbb{P}(\underline{\alpha}, \xi)$.

\begin{proposition}\label{circlocprop}
Consider $\mathbb{P}(\underline{\alpha}, \xi)$ such that $\alpha_i < 0$ for at least one $i$,
$D \in A_{d - 1}\big(\mathbb{P}(\underline{\alpha}, \xi)\big)$ and $V$ the maximal complete
torus invariant subvariety of $\mathbb{P}(\underline{\alpha}, \xi)$, then:
\begin{enumerate}[(i)]
\item\label{circlocpropi} $H^d_V\big(\mathbb{P}(\underline{\alpha}, \xi),
\sh{O}(D)\big) \neq 0$;
\item\label{circlocpropii} $H_V^{\vert\ocircuit^-\vert}\big(
\mathbb{P}(\underline{\alpha}, \xi), \sh{O}(D)\big) = 0$ iff $D \in
F_{\ocircuit}$;
\item\label{circlocpropiii} $H_V^{i}\big(\mathbb{P}(\underline{\alpha}, \xi),
\sh{O}(D)\big) = 0$ for all $i \neq d$, $\vert\ocircuit^-\vert$.
\end{enumerate}
\end{proposition}

\begin{proof}
Consider first $I = \circuit$, then $(\hat{\Delta}_\ocircuit)_I =
\hat{\Delta}_\ocircuit \cong B^{d  - 1}$ and $(\hat{\Delta}_\ocircuit)_{V, I} =
(\hat{\Delta}_\ocircuit)_V \cong S^{d - 2}$. It follows that
$H^i\big(\hat{\Delta}_\ocircuit; k\big) = 0$ for all $i$ and
$H^{i - 1}\big(\hat{\Delta}_\ocircuit, (\hat{\Delta}_\ocircuit)_V; k\big) \cong
H^{i - 2}\big((\hat{\Delta}_\ocircuit)_V; k\big)$. As by assumption,
$\ocircuit^+ \neq \circuit$, so the associated hyperplane arrangement contains
an unbounded chamber such that $l_i(m) \geq -c_i$ for all $i \in \circuit$ and all
$m$ in this chamber. Hence (\ref{circlocpropi}) follows.
As in the proof of lemma \ref{circtoplemma}, it follows that $\hat{\Delta}_I$ is
contractible whenever $\ocircuit^+ \cap I \neq \emptyset$ and $\ocircuit^- \cap I
\neq \emptyset$. So in that case $H^i(\hat{\Delta}_I) = 0$ for all $i$ and
$H^{i - 1}(\hat{\Delta}_I, \hat{\Delta}_{V, I}; k) = H^{i - 2}(\hat{\Delta}_{V, I};
k)$ for all $i$.

Now let $I = \ocircuit^+$; then $(\hat{\Delta}_\ocircuit)_I =
(\hat{\Delta}\ocircuit)_{V, I} \cong S^{\ocircuit^+ - 2}$, so
$H^i\big((\hat{\Delta}_\ocircuit)_I, (\hat{\Delta}\ocircuit)_{V, I}; k\big) = 0$
for all $i$.
For $I =
\ocircuit^-$, then $(\hat{\Delta}_\ocircuit)_I \cong B^{\vert \ocircuit^- \vert - 1}$
and $(\hat{\Delta}_\ocircuit)_{V, I} \cong S^{\vert \ocircuit^- \vert - 2}$, the former
by Lemma \ref{circtoplemma}, the latter by Lemma \ref{circtoplemma} and the fact that
$(\hat{\Delta}_\ocircuit)_{V, I}$ has empty intersection with $\stern(\tau)$.
This implies (\ref{circlocpropii}) and consequently (\ref{circlocpropiii}).
\end{proof}

\subsection{Nef cone and Kawamata-Viehweg vanishing, intersection numbers}
\label{onecircuitnefkawamataviehweg}

The nef cone of $\mathbb{P}(\underline{\alpha}, \xi)$ is given by the half line
$[0, \infty)$ in $A_{d - 1}\big(\mathbb{P}(\underline{\alpha}, \xi)\big)_\Q$,
the ample cone by its interior.

\begin{definition}
We denote $K_\ocircuit := - \sum_{i \in \ocircuit^+} D_i$ the {\em minimal}
divisor and $A_\nef := K_\ocircuit + (0, \infty) = (-\sum_{i \in \ocircuit^+} D_i, \infty)$.
\end{definition}

If we identify $K_\ocircuit$ with its class $\sum_{i \in \ocircuit^+} \alpha_i$
in $A_{d - 1}\big(\mathbb{P}(\underline{\alpha}, \xi)\big)_\Q$, we obtain that
$K_\ocircuit \leq D$ for every $D = \sum_{i \in \circuit}
e_i D_i$ for with $-1 \leq e_i < 0$ for all $i \in \circuit$. In particular,
$K_{\mathbb{P}(\underline{\alpha}, \xi)} = -\sum_{i \in \circuit} D_i \geq K_\ocircuit$
with equality if and only if $\ocircuit^+ = \circuit$.
This is the Kawamata-Viehweg theorem for toric $1$-circuit varieties:

\begin{proposition}\label{circuitkawamataviehweg}
Let $D := D' + E$ be a Weil divisor on $\mathbb{P}(\underline{\alpha}, \xi)$, where $D'$
is $\Q$-ample and $E = \sum_{i \in \circuit} e_i D_i$ with $-1 \leq e_i < 0$ for all
$i \in \circuit$. Then $H^i\big(\mathbb{P}(\underline{\alpha}, \xi), \sh{O}(D)\big) = 0$
for all $i > 0$.
\end{proposition}

\begin{proof}
The assertion follows from \ref{circuitglobalcohom}. Note that every integral divisor class
which maps to $A_\nef$ is already contained in $F_\ocircuit$.
\end{proof}

We point out that the converse of \ref{circuitkawamataviehweg} is not true
in general. As already illustrated in the introduction, there might be Weil
divisors which do not map to a class in $A_\nef$ and which have no cohomology.
As the theorem is stated over $\Q$, it only captures the offsets of the
diophantine Frobenius problem. However, the theorem is sharp if $L$ contains
a $\Z$-basis of $N$, in particular if $\mathbb{P}(\underline{\alpha}, \xi)$
is smooth. Also note that $\sh{O}(K_\ocircuit)$ by Proposition \ref{circuitglobalcohom}
has nonvanishing cohomology.

\begin{definition}
let $D = \sum_i c_i D_i$ be a $T$-invariant divisor of $\mathbb{P}(\underline{\alpha}, \xi)$. Then
we denote
\begin{equation*}
P_D := \{m \in M_\Q \mid l_i(m) \geq -c_i\}.
\end{equation*}
\end{definition}

Let $D$ be an effective invariant Cartier divisor on $\mathbb{P}(\underline{\alpha},
\xi)$ and let $\tau \in \Delta_\ocircuit(d - 1)$ an inner wall. The intersection
product $D . V(\tau)$ can be read off the polyhedron $P_D$. Namely, let $P_{D, \tau}
\subset P_D$ the one-dimensional face of $P_D$ corresponding to $\tau$. Then $P_{D,
\tau}$ is a bounded interval in $M_\Q$ and it is easy to see that the number of
lattice points on $P_{D, \tau}$ coincides with $\dim H^0\big(V(\tau), \sh{O}(D)
\vert_{V(\tau)} \big) = 1 + \deg \sh{O}(D)\vert_{V(\tau)}$. Let $\{i, j\} = \circuit
\setminus \tau(1)$. If we first restrict to the open subvariety $U :=
U_{\sigma_i} \cup U_{\sigma_j}$ of $\mathbb{P}(\underline{\alpha}, \xi)$, we obtain
analogously to proposition \ref{circuitpic} that $\pic(U)$ is generated by
$\frac{s}{s_\tau} \alpha_i^{-1} \cdot \lcm\{\alpha_i, \alpha_j\} \cdot D_i$,
which corresponds to the class in $\pic(U)$ such
that $P_{D, \tau}$ has lattice length $1$. Over $\Q$, we can identify $D_i$
with its class $\alpha_i$ in $A_{d - 1}\big(\mathbb{P}(\underline{\alpha}, \xi)\big)_\Q
\cong \Q$, and we obtain:
\begin{equation*}
D_i . V(\tau) = \frac{s_\tau}{s} \alpha_j^{-1} \gcd\{\alpha_i, \alpha_j\}.
\end{equation*}
Note that here, and for the rest of this work, we are only interested in intersections
of curves with divisors. For simplicity, we will not distinguish between cycles
and their degree.

The inclusion $U \hookrightarrow \mathbb{P}(\underline{\alpha}, \xi)$ does not
change the intersection product, so above formula also holds on
$\mathbb{P}(\underline{\alpha}, \xi)$.
For any other $k \in \circuit$, we then obtain by linearity that
$D_k . V(\tau) = \frac{\alpha_k}{\alpha_i} D_i . V(\tau)$. So we get a handy formula
for the linear form on $A_{d - 1}\big(\mathbb{P}(\underline{\alpha}, \xi)\big)_\Q$
associated to $\tau$:

\begin{definition}
We set
\begin{equation*}
t_{\ocircuit, \tau} := t_{\underline{\alpha}, \xi, \tau} :=
\frac{s_\tau}{s} \frac{\gcd\{\alpha_i, \alpha_j\}}{\alpha_i \alpha_j}
= \frac{s_\tau}{s} \lcm\{\alpha_i, \alpha_j\}^{-1}. 
\end{equation*}
\end{definition}

\subsection{Fujita's ampleness theorem for toric 1-circuit varieties.}

Recall that a Cartier divisor $D$ is ample if $\Delta$ coincides with the
inner normal fan of $P_D$. The divisor $D$ is very ample if it is ample and
moreover, for every $\sigma \in \Delta$,
the semigroup $\sigma_M$ is generated by $(P_D \cap M) - m_\sigma$, where $m_\sigma$
is the corner of $P_D$ corresponding to $\sigma$. We call a $\Q$-divisor ample, if some
integral multiple of it is ample.
We want to give an ampleness criterion similar to proposition
\ref{circuitkawamataviehweg} in terms of intersection numbers.
First we give an analog to a theorem of Fujino \cite{Fujino03} (see also
\cite{payne06a}):

\begin{proposition}\label{circuitfujino}
Let $E = \sum_{i \in \circuit} e_i D_i$, where $-1 \leq e_i \leq 0$ for all $i \in
\circuit$ and $D$ an invariant $\Q$-divisor on $\mathbb{P}(\underline{\alpha}, \xi)$.
If for every inner
wall $\tau$ of $\Delta_\ocircuit$ we have $D . V(\tau) \geq \vert \ocircuit^+ \vert
- x$, where $x = 0$ if $\mathbb{P}(\underline{\alpha}, \xi)$ is smooth and $x = 1$
otherwise, then $D + E$ is nef.
\end{proposition}

\begin{proof}
For every inner wall $\tau$, we have $D . V(\tau) = D . t_{\ocircuit, \tau} =
 D \cdot \frac{s_\tau}{s}
\lcm\{\alpha_i, \alpha_j\}^{-1} \geq \vert \ocircuit^+ \vert - x$, where
$\{i, j\} = \circuit \setminus \tau(1)$. Thus $D \geq \frac{s}{s_\tau}
\lcm\{\alpha_i, \alpha_j\} \cdot (\vert \ocircuit^+ \vert - x)$. Hence
$D \geq \max\{\alpha_i\}_{i \in \ocircuit^+} \cdot (\vert \ocircuit^+ \vert - x)
\geq K_\ocircuit$, with equality if and only if $\mathbb{P}(\underline{\alpha}, \xi)$
is smooth.
\end{proof}

We obtain some criteria for ampleness in the spirit of the toric Fujita ampleness
theorem (see also \cite{Mustata1}, \cite{payne06a}):

\begin{corollary}\label{circuitfujita}
Let $\mathbb{P}(\underline{\alpha}, \xi)$ and $E$, $D$ be as before, such that
$D + E$ Cartier. If for every inner wall $\tau$ of $\Delta_\ocircuit$
\begin{enumerate}[(i)]
\item\label{circuitfujitaii} $D . V(\tau) \geq \vert \ocircuit^+ \vert + 1$ and
$\mathbb{P}(\underline{\alpha}, \xi)$ is smooth, then $D + E$ is very ample;
\item\label{circuitfujitai} $D . V(\tau) \geq d + 1$ and
$\mathbb{P}(\underline{\alpha}, \xi)$ is not smooth, then $D + E$ is very ample;
\item\label{circuitfujitaiii} $D . V(\tau) \geq \vert \ocircuit^+ \vert$ and
$\mathbb{P}(\underline{\alpha}, \xi)$ is not smooth, then $D + E$ is ample.
\end{enumerate}
\end{corollary}

\begin{proof}
For (\ref{circuitfujitaii}) note that all $D_i$ with $i \in \ocircuit^+$ are linearly
equivalent and $D_i . V(\tau) = 1$ for every inner wall $\tau$ and all $i \in \ocircuit^+$.
Thus $D > \sum_{i \in \ocircuit^+} D_i = -K_\ocircuit$ and thus $D + E$ is effective and
nef, and therefore ample by smoothness of $\mathbb{P}(\underline{\alpha}, \xi)$.

If $\mathbb{P}(\underline{\alpha}, \xi)$ is not minimal, we need sufficient conditions
for $D + E$ to be ample or very ample. To show that $D + E$ is
very ample, we have to verify that for every $i \in \ocircuit^+$ with associated maximal cone
$\sigma_i$, the shifted polytope $P_{D + E} - m_{\sigma_i}$ generates the semigroup
$\sigma_{i, M}$, where $m_{\sigma_i} \in M$ is the corner of $P_{D + E}$ corresponding to
$\sigma_i$. Let $p_1, \dots, p_d$ be the primitive vectors of the rays of the dual cone
$\check{\sigma}_i$.
Due to a criterion of Ewald and Wessels \cite{EwaldWessels},
it suffices to show that the simplex spanned by $0$ and $(d - 1)
\cdot p_j$ is contained in $P_{D + E} - m_{\sigma_i}$.
Note that the $p_j$, where $j \in \ocircuit^-$, span a subcone of the recession
cone of $P_{D + E}$, i.e. $\N . p_j \subset P_{D + E} - m_{\sigma_i}$. So we
have only to check the $p_j$ with $j \in \ocircuit^+$.
As we have seen before, the lattice distance of $p_j$ is given by $\frac{s}{s_\tau}
\alpha_i^{-1} \lcm(\alpha_i, \alpha_j)$. Let $\tau_{i, j} := \sigma_i \cap \sigma_j$, then
we obtain $D . V(\tau_{i, j}) \geq \frac{s}{s_\tau} \lcm(\alpha_i, \alpha_j) \cdot
(d + 1)$ and thus $D$ is at least the factor $\alpha_i \cdot
(d + 1)$ larger than the minimal divisor $D_{ij}$ such that
$P_{D_{ij}, \tau}$ has lattice length at least one. In fact, $D \geq \max\{\alpha_i \cdot
(d + 1) \cdot D_{ij}\}_{j \in \ocircuit^+ \setminus \{i\}}$.
Thus we get that the lattice length of $P_{D + E, \tau_{ij}} \geq d + 1$ and
(\ref{circuitfujitai}) follows. Assertion (\ref{circuitfujitaiii}) follows at once
as from the estimate follows that $p_i \in P_{D + E} - m_\sigma$ for every $i \in \circuit$.
\end{proof}

\subsection{Cohomology and resolutions of singularities.}

It is instructive to see the local cohomology vanishing in the context of classification
of maximal Cohen-Macaulay modules. Assume that $\vert\ocircuit^-\vert \notin \{0, 1\}$,
then the
$l_i$ span a strictly convex cone which gives rise to an affine toric variety $Y$.
Recall that there is a natural map $\pi: \mathbb{P}(\underline{\alpha}, \xi)
\longrightarrow  Y$ which is a small resolution of singularities.
Likewise, by flipping we obtain a
second resolution $\pi': \mathbb{P}(-\underline{\alpha}, \xi)) \longrightarrow Y$.
We have two natural isomorphisms $A_{d - 1}(Y) \overset{\pi^{-1}}{\longrightarrow}
A_{d - 1}(\mathbb{P}(\underline{\alpha}, \xi))$ and $A_{d - 1}(Y) \overset{(\pi')^{-1}}{\longrightarrow}
A_{d - 1}(\mathbb{P}(\underline{\alpha}, \xi))$, which both are
induced by the identity on $\Z^n$. These isomorphisms map any
Weil divisor $D$ on $Y$ to its strict transforms $\pi^{-1} D$ or $(\pi')^{-1} D$, respectively,
on $\mathbb{P}(\underline{\alpha}, \xi))$ or $\mathbb{P}(-\underline{\alpha}, \xi))$, respectively.

Now, the question whether $\sh{O}(D)$ is a maximal Cohen-Macaulay
sheaf can be decided directly on $Y$ or, equivalently, on the resolutions:

\begin{theorem}
Let $Y$ be an affine toric variety whose associated cone $\sigma$ is spanned by a circuit
$\circuit$ and denote $\mathbb{P}(\underline{\alpha}, \xi)$ and $\mathbb{P}(-\underline{\alpha}, \xi)$
the two canonical small resolution of singularities.
Then the sheaf $\sh{O}(D)$ is maximal Cohen-Macaulay if and only if $R^i\pi_* \sh{O}
(\pi^{-1}D) = R^i\pi'_* \sh{O}((\pi')^{-1}D) = 0$ for all $i > 0$.
\end{theorem}

\begin{proof}
This toric variety corresponds to the toric subvariety of $Y$ which is the complement
of its unique fixed point, which we denote $y$. We have to show that $H^i_y\big(Y, \sh{O}(D)\big)
= 0$ for all $i < d$. By Corollary \ref{localcohomcorollary}, we have
\begin{equation*}
H^i_y\big(Y, \sh{O}(D)\big)_m = H^{i - 2}(\hat{\sigma}_{y, m}; k)
\end{equation*}
for every $m \in M$, where $\hat{\sigma}_y$ denotes the simplicial model for the fan
associated to $Y \setminus \{y\}$.
Denote $\tau$ and $\tau'$ the cones corresponding to the minimal orbits of
$\mathbb{P}(\underline{\alpha}, \xi)$ and $\mathbb{P}(-\underline{\alpha}, \xi)$,
respectively.
We observe that $(\hat{\Delta}_\ocircuit)_{V(\tau)} = (\hat{\Delta}_\ocircuit)_{V(\tau')}$
both coincide with the subfan of $\sigma$ generated by its facets. It follows that
the simplicial complexes relevant for computing the isotypical decomposition
of $H^i_y\big(Y, \sh{O}(D)\big)$ coincide with the simplicial complexes
relevant for computing the $H^i_V\big(\mathbb{P}(\underline{\alpha}, \xi), \sh{O}(\pi^{-1}D)\big)$
and $H^i_{V'}\big(\mathbb{P}(-\underline{\alpha}, \xi), \sh{O}((\pi')^{-1}D)\big)$,
respectively, where $V, V'$ denote the exceptional sets of the morphisms $\pi$ and
$\pi'$, respectively. By Proposition \ref{circlocprop} the corresponding cohomologies
vanish for $i < d$
iff $D \in F_\ocircuit \cap F_{-\ocircuit}$. Now we observe that $\Gamma\big(Y, R^i\pi_*
\sh{O}(\pi^{-1}D)\big) = H^i\big(\mathbb{P}(\underline{\alpha}, \xi), \sh{O}(\pi^{-1}D)\big)$
and $\Gamma\big(Y, R^i\pi'_* \sh{O}((\pi')^{-1}D)\big) = H^i\big(\mathbb{P}(-\underline{\alpha},
\xi), \sh{O}((\pi')^{-1}D)\big)$. By Proposition \ref{circuitglobalcohom}, both cohomologies
vanish for $i > 0$ iff $D \in F_\ocircuit \cap F_{-\ocircuit}$.
\end{proof}

\begin{remark}
The relation between maximal Cohen-Macaulay modules and the diophantine Frobenius
problem has also been discussed in \cite{Stanley96}. See \cite{Yoshino90} for a
discussion of MCM-finiteness of toric $1$-circuit varieties.
\end{remark}

\begin{remark}
The fiber of $\pi$ over the exceptional locus again is a toric $1$-circuit variety,
 a finite quotient of
a weighted projective space. This variety is given
by $\mathbb{P}(\underline{\alpha}^+, \xi^+)$, whose associated fan is contained in
$N / \overline{N}_{\ocircuit^-}$. Here, $\xi^+: N / \overline{N}_{\ocircuit^-} \longrightarrow
N / \overline{N}_{\ocircuit^-}$ is the morphism induced by $\xi$, and
for $\alpha^+$, the $\alpha^+_i$ equal to $\alpha_i$ divided by the greatest
common divisor of all $\lcm\{\alpha_j\}_{j \in I}$, where $I$ runs over all maximal
proper subsets of $\ocircuit^+$ which contain $i$.
\end{remark}

\section{Discriminants and combinatorial aspects cohomology vanishing}\label{discriminantalarrangementsandsecfans}

In this section we will concentrate on aspects of toric geometry which are related to its underlying
linear algebra. A toric variety $X$ is specified by the set of primitive vectors $l_1, \dots, l_n \in N$ and the
fan $\Delta$ supported on these vectors. We can separate three properties which
govern the geometry of $X$ and are relevant for cohomology vanishing problems:
\begin{enumerate}[(i)]
\item\label{aspecti} the linear algebra given by the vectors $l_1, \dots, l_n$ and their linear dependencies
as $\Q$-vectors;
\item\label{aspectii} arithmetic properties, which are also determined by the $l_i$, but considered as integral vectors;
\item\label{aspectiii} its combinatorics, which is given by the fan $\Delta$.
\end{enumerate}

In the case of toric $1$-circuit varieties it was possible to study the arithmetic aspects in isolation. For
general toric varieties, this is no longer possible, because the three properties interact in much more complicated
ways, which we have to keep track of. In this section we will describe how linear algebraic and combinatorial
aspects are combined. In the sequel it will be convenient to consider the $l_i$ as matrix.
So we define:
\begin{definition}
For any given set of vectors $l_1, \dots, l_n \in N_\Q$ we denote $L$ the matrix whose rows are given
by the $l_i$. For any subset $I$ of $\on$ we denote $L_I$ the submatrix of $L$ whose rows are given
by the $l_i$ with $i \in I$.
\end{definition}
We will also frequently make use of the following abuse of notion:
\begin{convention}
We will usually identify subsets $I \subset \on$ with the corresponding
subsets of $\{l_1, \dots, l_n\}$. In particular, if $\circuit \subset \on$ such that the
set $\{l_i\}_{i \in \circuit}$ forms a circuit, then we will also call $\circuit$ a circuit.
Also, we will in general not distinguish between $\{l_i\}_{i \in I}$ and $L_I$.
\end{convention}

\begin{definition}
We say that a fan $\Delta$ is {\em supported on} $l_1, \dots, l_n$ if $\Delta(1)$ coincides with
the set of rays generated by a subset of the $l_i$.
\end{definition}

Moreover, note that
circuits are not required to span $M_\Q$. If some circuit $\circuit \subset \on$
generates a subvector space of codimension $r$ in $N_\Q$, then
for some orientation $\ocircuit$ of $\circuit$ the variety $X(\Delta_\ocircuit)$
is isomorphic to $\mathbb{P}(\underline{\alpha}, \xi) \times (k^*)^r$ for some
appropriate $\underline{\alpha}$ and $\xi$. If $\ocircuit^+ = \circuit$,
the fan $\Delta_{-\ocircuit}$ is empty. By convention, in that case we define
$X(\Delta_{-\ocircuit}) := (k^*)^r$ as the associated toric variety.
The relevant facts from section
\ref{onecircuitvarieties} can straightforwardly be adapted to this situation.

\begin{definition}
We denote $\circuit(L)$ the set of circuits of $L$ and $\ocircuit(L)$ the set of oriented circuits
of $L$, i.e. the set of all orientations $\ocircuit, -\ocircuit$ for $\circuit \in \circuit(L)$.
\end{definition}

In subsection \ref{circuitsdiscriminantal} we consider circuits of the matrix $L$ and the
induced stratification of $A_{d - 1}(X)_\Q$. In subsection \ref{secondaryfans}
we will collect some well-known material on secondary fans from \cite{GKZ}, \cite{OdaPark},
and \cite{BilleraFillimanSturmfels} and explain their relation to
discriminantal arrangements. Subsection \ref{birat} then applies this to certain
statements about the birational geometry of toric varieties and cohomology vanishing.

\subsection{Circuits and discriminantal arrangements}\label{circuitsdiscriminantal}

Recall that for any torus invariant divisor $D = \sum_{i \in \on} c_i D_i$, the isotypical components
$H^i_V\big(X, \sh{O}(D)\big)_m$ for some cohomology group depend on simplicial complexes
$\hat{\Delta}_I$, where $I = I(m) = \{i \in \on \mid l_i(m) < -c_i\}$. So, the set of all possible
subcomplexes $\hat{\Delta}_I$ depends on the chamber decomposition of $M_\Q$ which is induced
by the hyperplane arrangement which is given by hyperplanes $H_1, \dots, H_n$, where
\begin{equation*}
H_i^\uc := \{m  \in M_Q \mid l_i(m) = -c_i\}.
\end{equation*}
The set of all relevant $I \subset \on\}$ is determined by the map
\begin{equation*}
\mathfrak{s}^{\underline{c}} : M_\Q \longrightarrow 2^\on, \quad m \mapsto \{i \in \on \mid  l_i(m) < -c_i\}.
\end{equation*}

\begin{definition}
For $m \in M_\Q$,  we call $\mathfrak{s}^{\underline{c}}$ the {\em signature} of $m$.
We call the image of $M_\Q$ in $2^\on$ the {\em combinatorial type} of $\uc$.
\end{definition}

\begin{remark}
The combinatorial type encodes what in combinatorics is known as {\em oriented matroid}
(see \cite{BLSWZ}). We will not make use of this kind of structure, but we will find
it sometimes convenient to borrow some notions.
\end{remark}

So, given $l_1, \dots, l_n$, we would like to classify all possible combinatorial types, depending
on $\underline{c} \in \Q^n$. The natural parameter space for all hyperplane arrangements up
to translation by some element $m \in M_Q$  is given by the set $A_\Q \cong \Q^n / M_\Q$,
which is given by following short exact sequence:
\begin{equation*}
0 \longrightarrow M_\Q \overset{L}{\longrightarrow} \Q^n \overset{D}{\longrightarrow}
A_{d - 1}(X)_\Q = A_\Q \longrightarrow 0.
\end{equation*}
Then the $D_1, \dots, D_n$ are the images of the standard basis vectors of $\Q^n$.
This procedure of constructing the $D_i$ from the $l_i$ is often called {\em Gale transformation},
and the $D_i$ are the {\em Gale duals} of the $l_i$.

Now, a hyperplane arrangement $H^\uc_i$ for some $\uc \in \Q^n$,
is considered in general position if the hyperplanes $H_i^\uc$ intersect in the smallest possible dimension.
When varying $\uc$ and passing from one arrangement in general position to another with of
a different combinatorial type, this necessarily implies that has to take place some specialization
for some $\uc \in \Q^n$, i.e. the corresponding hyperplanes $H_i^\uc$ do not intersect in the
smallest possible dimension. So we see that the combinatorial types of hyperplane arrangements
with fixed $L$ and varying induce a stratification of $A_\Q$, where the maximal strata correspond
to hyperplane arrangements in general position. The determination of this stratification is the
{\em discriminant problem} for hyperplane arrangements. To be more precise, let $I \subset \on$
and denote
\begin{equation*}
H_I := \{\uc + M_\Q \in A_\Q \mid \bigcap_{i \in I} H_i^\uc \neq 0\},
\end{equation*}
i.e. $H_I$ represents the set of all hyperplane arrangements (up to translation) such that the hyperplanes
$\{H_i\}_{i \in I}$ have nonempty intersection. The sets $H_I$ can be described straightforwardly by
the following commutative exact diagram:
\begin{equation}\label{HIdefdiagram}
\UseComputerModernTips
\xymatrix{
& & & H_I \ar@^{{(}->}[d] \\
0 \ar[r] & M_\Q \ar[r]^L \ar@{=}[d] & \Q^n \ar[r]^D \ar@{-{>>}}[d] & A_\Q \ar[r] \ar@{-{>>}}[d] &
0 \\
& M_\Q \ar[r]^{L_I} & \Q^I \ar[r]^{D_I} &
A_{I, \Q} \ar[r]& 0.
}
\end{equation}
In particular, $H_I$ is a subvector space of $A_\Q$. Moreover, we immediately read off
diagram (\ref{HIdefdiagram}):

\begin{lemma}
\label{dimlemma}
\begin{enumerate}[(i)]
\item\label{dimlemmaiv} $H_I$ is generated by the $D_i$ with $i \in \on \setminus I$.
\item\label{dimlemmai} $\dim H_I = n - \vert I \vert - \dim ( \ker L_I)$.
\item\label{dimlemmaii} If $J \subseteq I$ then $H_I \subseteq H_J$.
\item\label{dimlemmaiii} Let $I, J \subset \on$, then $H_{I \cup J} \subset H_I \cap H_J$.
\end{enumerate}
\end{lemma}

Note that in (\ref{dimlemmaiii}), the reverse inclusion in
general is not true. It follows that the hyperplanes among the $H_I$ are
determined by the formula:
\begin{equation*}
\vert I \vert = \rk L_I + 1.
\end{equation*}
By Lemma \ref{dimlemma} (\ref{dimlemmaii}), we can always consider the
minimal linearly dependent sets $I$, i.e. circuits, fulfilling this condition.
It turns out that the hyperplane $H_\circuit$ suffice to completely
describe the discriminants of $L$:

\begin{lemma}
\label{intersectionproperty}
Let $I \subset \on$, then
\begin{equation*}
H_I = \bigcap_{\circuit \subset I \text{ circuit}} H_\circuit,
\end{equation*}
where, by convention, the right hand side equals $A_\Q$, if the $l_i$ with $i \in I$ are linearly independent.
\end{lemma}

Hence, the stratification of $A_\Q$ which we were looking for is completely determined
by the hyperplanes $H_\circuit$.

\begin{definition}
We denote the set $\{H_\circuit \mid \circuit \subset \on \text{ a circuit}\}$
the {\em discriminantal arrangement} of $L$.
\end{definition}

\begin{remark}
The discriminantal arrangement carries a natural matroid structure. This structure
can be considered as another combinatorial invariant of $L$ (or the toric variety $X$,
respectively), its {\em circuit geometry}.
Discriminantal arrangements seem to have been appeared first in \cite{Crapo84},
where the notion of 'circuit geometry' was coined. The notion of discriminantal
arrangements stems from \cite{ManinSchechtman89}. Otherwise, this subject seems
to have been studied explicitly only in very few places, see for instance \cite{Falk94},
\cite{BayerBrandt}, \cite{Athanasiadis99}, \cite{Reiner99}, \cite{Cohen00},
\cite{CohenVarchenko},
though it is at least implicit in the whole body of literature on secondary fans.
Above references are mostly concerned with genericity properties of discriminantal
arrangements. Unfortunately, in toric geometry, the most interesting cases (such as
smooth projective toric varieties, for example) virtually never give rise to
discriminantal
arrangements in general position. Instead, we will focus on certain properties of
nongeneric circuit geometries, though we will not undertake a thorough combinatorial
study of these.
\end{remark}

Virtually all problems related to cohomology vanishing on a toric variety
$X$ must depend on the associated discriminantal arrangement and therefore on the
circuits of $L$. In subsection \ref{secondaryfans} we will see that the discriminantal
arrangement is tightly tied to the geometry of $X$.

As we have seen in section \ref{onecircuitvarieties}, to every circuit $\circuit \subset \on$ we
can associate two oriented circuits. These correspond to the signature of the bounded chamber
of the subarrangement in $M_\Q$ given by the $H_i^\uc$ with $i \in \circuit$ (or better to
the bounded chamber in $M_\Q / \ker L_I$, as we do no longer require that the $l_i$ with $i
\in \circuit$ span $M_\Q$). Lifting this to $A_\Q$, this corresponds to the half spaces
in $A_\Q$ which are bounded by $H_\circuit$.

\begin{definition}
Let $\circuit \subset \on$ be a circuit, then we denote $H_\ocircuit$ the half space in $A_\Q$
bounded by $H_\circuit$ corresponding to the orientation $\ocircuit$.
\end{definition}

We obtain immediately:

\begin{lemma}
Let $\circuit$ be a circuit of $L$ and $\ocircuit$ an orientation of $\circuit$. Then the
hyperplane $H_{\circuit}$ is {\em separating}, i.e. for every $i \in \on$ one of the
following holds:
\begin{enumerate}[(i)]
\item $i \in \on \setminus \circuit$ iff $D_i \in H_\circuit$;
\item if $i \in \ocircuit^+$, then $D_i \in H_\ocircuit \setminus H_\circuit$;
\item if $i \in \ocircuit^-$, then $D_i \in H_{-\ocircuit} \setminus H_\circuit$.
\end{enumerate}
\end{lemma}

Now we are going to borrow some terminology from combinatorics. Consider any subvector
space $U$ of $A_\Q$ which is the intersection of some of the $H_\circuit$. Then the set
$\mathcal{F}_U$ of all $\circuit \in \circuit(L)$ such that $H_\circuit$ contains $U$ is called a
{\em flat}. The subvector space is uniquely determined by the flat and vice versa. We can
do the same for the actual strata rather than for subvector spaces. For this, we just need
to consider instead the oriented circuits and their associated half spaces in $A_\Q$: any
stratum $S$ of the discriminantal arrangement uniquely determines a finite set $\mathfrak{F}_S$
of oriented circuits $\ocircuit$ such that $S \subset H_\ocircuit$. From the set $\mathfrak{F}_S$
we can reconstruct the closure of $S$:
\begin{equation*}
\overline{S} = \bigcap_{\ocircuit \in \mathfrak{F}_\mathcal{S}} H_\ocircuit,
\end{equation*}
We give a formal definition:

\begin{definition}
For any subvector space $U \subset A_\Q$ which is a union of strata of the discriminantal
arrangement, we denote $\mathcal{F}_U := \{\circuit \in \circuit(L) \mid U \subset H_\circuit\}$
the associated {\em flat}.
For any single stratum $S \subset A_\Q$ of the discriminantal arrangement, we denote
$\mathfrak{F}_S := \{\ocircuit \in \ocircuit(L) \mid U \subset H_\ocircuit\}$ the
associated  {\em oriented flat}.
\end{definition}

The notion of flats gives us some flexibility in handling strata. Note that flats
reverse inclusions, i.e. $S \subset T$ iff $\mathfrak{F}_T \subset \mathfrak{F}_S$.
Moreover, if a stratum $S$
is contained in some $H_\circuit$, then its oriented flat contains both $H_\ocircuit$
and $H_{-\ocircuit}$, and vice versa. So from the oriented flat we can reconstruct
$\mathcal{F}_\circuit$ and thus the subvector space of $A_\Q$ generated by $S$.

\begin{definition}
Let $\mathcal{S} := \{S_1, \dots, S_k\}$ be a collection of strata of the discriminantal
arrangement. We call
\begin{equation*}
\mathfrak{F}_\mathcal{S} := \bigcap_{i = 1}^k \mathfrak{F}_{S_i} 
\end{equation*}
the {\em discriminantal hull} of $\mathcal{S}$.
\end{definition}

The discriminantal hull defines a closed cone in $A_\Q$ which is given by the intersection
$\bigcap_{\ocircuit \in \mathfrak{F}_\mathcal{S}} H_\ocircuit$. This cone contains the
union of the closures $\overline{S}_i$, but is bigger in general.

\begin{lemma}\label{stratalemma}
\begin{enumerate}[(i)]
\item\label{stratalemmai} Let $\mathcal{S} = \{S_1, \dots, S_k\}$ be a collection of
discriminantal strata whose union is a closed cone in $A_\Q$. then
$\mathfrak{F}_\mathcal{S} = \bigcap_{i = 1}^k \mathfrak{F}_{S_i}$.
\item\label{stratalemmaii} Let $\mathcal{S} = \{S_1, \dots, S_k\}$ be a collection of
discriminantal strata and $U$ the subvector space of $A_\Q$ generated by the
$S_i$. Then the forgetful map $\mathfrak{F}_\mathcal{S} \rightarrow \mathcal{F}_U$
is surjective iff $\mathfrak{F}_\mathcal{S} = \mathfrak{F}_{S_i}$ for some $i$.
\end{enumerate}
\end{lemma}

\begin{proof}
For (\ref{stratalemmai}) just note that because $\bigcup_{i = 1}^k S_k$ is a closed
cone, it must be an intersection of some $H_\ocircuit$.
For (\ref{stratalemmaii}): the set $\bigcap_{\ocircuit \in \mathfrak{F}_\mathcal{S}}
H_\ocircuit$ is a cone which contains the convex hull of all the $\overline{S}_i$.
If some $\circuit$ is not in the image of the forgetful map, then the hyperplane
$H_\circuit$ must intersect the relative interior of this cone. So the assertion follows.
\end{proof}

\subsection{Secondary Fans}
\label{secondaryfans}

For any $\uc \in \Q^n$ the arrangement $H_i^\uc$ induces a chamber decomposition of $M_\Q$, where
the closures of the chambers are given by
\begin{equation*}
P_\uc^I := \{m \in M_\Q \mid l_i(m) \leq -c_i \text{ for } i \in I \text{ and } l_i(m) \geq -c_i \text{ for } i \notin I\}
\end{equation*}
for every $I \subset \on$ which belongs to the combinatorial type of $\uc$. In particular, $\uc$ represents
an element $D \in A_\Q$  with
\begin{equation*}
D \in \bigcap_{I \in \mathfrak{s}^\uc(M_\Q)} C_I,
\end{equation*}
where $C_I$ is the cone in $A_\Q$ which is generated by the $-D_i$ for $i \in I$ and the $D_i$ with $i \notin I$
for some $I \subset \on$.
For an invariant divisor $D = \sum_{i \in \on} c_i D_i$ we will also write
$P_D^I$ instead of $P_\uc^I$. If $I = \emptyset$, we will occasionally omit the index $I$.

The faces of the $C_I$ can be read off directly from the signature:

\begin{proposition}
\label{orthantboundary}
Let $I \subset \on$, then $C_I$ is an nonredundant intersection of the $H_\ocircuit$
with $\ocircuit^- \subset I$ and $\ocircuit^+ \cap I = \emptyset$.
\end{proposition}

\begin{proof}
First of all, it is clear that $C_I$ coincides with the intersection of
half spaces
\begin{equation*}
C_I = \bigcap_{\substack{\ocircuit^+ \subset I\\ \ocircuit^- \cap I = \emptyset}}
H_\ocircuit.
\end{equation*}
For any $H_\ocircuit$ in the intersection let $H_\circuit$ its boundary. Then
$H_\circuit$ contains a cone of codimension $1$ in $A_\Q$ which is spanned by
$D_i$ with $i \in \on \setminus (\circuit \cup I)$ and by $-D_i$ with $i \in
I \setminus \circuit$ which thus forms a proper facet of $C_I$.
\end{proof}

Recall that the secondary fan of $L$ is a fan in $A_\Q$ whose maximal cones are in one-to-one correspondence
with the regular simplicial fans which are supported on the $l_i$. That is, if $\uc$ is chosen sufficiently
general, then the polyhedron $P_\uc^\emptyset$ is simplicial and its inner normal fan is a simplicial fan
which is supported on the $l_i$.
Wall crossing in the secondary fan then corresponds to a
transition $\Delta_\ocircuit \longrightarrow \Delta_{-\ocircuit}$ as in section \ref{onecircuitvarieties}.
Clearly, the secondary fan is a substructure of the discriminantal arrangement in the sense that its
cones are unions of strata of the discriminantal arrangements. However, the secondary fan in general
is much coarser than the discriminantal arrangement, as it only
keeps track of the particular chamber $P^\emptyset_\uc$. In particular, the secondary
fan is only supported on $C_\emptyset$ which in general does not coincide with $A_\Q$.
Of course, there is no reason to consider only one particular type of chamber --- we can consider
secondary fans for every $I \subset \on$ and every type of chamber $P^I_\uc$.
For this, observe first that, if
$\mathcal{B} $ is a subset of $\on$  such that the $l_i$ with $i \in \mathcal{B}$ form a basis of
$M_\Q$, then the complementary Gale duals $\{D_i\}_{i \notin \mathcal{B}}$ form a basis of $A_\Q$.
Then we set:

\begin{definition}
Let $I \subset \on $ and $\mathcal{B} \subset \on$ such that the $l_i$ with $i \in \mathcal{B}$
form a basis of $M_\Q$, then we denote $K^I_\mathcal{B}$ the cone in $A_\Q$ which is generated
by $-D_i$ for $i \in I \setminus \mathcal{B}$ and by $D_i$ for $i \in \on \setminus (I \cup \mathcal{B})$.
The secondary fan $\secfan(L, I)$ of $L$ with respect to $I$ is the fan whose
cones are the intersections of the $K^I_\mathcal{B}$, where $\mathcal{B}$ runs over all
bases of $L$.
\end{definition}

Note that $\secfan(L, \emptyset)$ is just the secondary fan as usually defined. Clearly, the chamber
structure of the discriminantal arrangement still refines the chamber structure induced by all secondary
arrangements. But now we have sufficient data to even get equality:

\begin{proposition}\label{chamberprop}
The following induce identical chamber decompositions of $A_\Q$:
\begin{enumerate}[(i)]
\item the discriminantal arrangement,
\item the intersection of all secondary fans $\secfan(L, I)$,
\item the intersection of the $C_I$ for all $I \subset \on$.
\end{enumerate}
\end{proposition}

\begin{proof}
Clearly, the facets of every orthant $C_I$ span a hyperplane which is part of
the discriminantal arrangement, so the chamber decomposition induced by the secondary
fan is a refinement of the intersection of the $C_I$'s. The $C_I$ induce a refinement
of the secondary fans as follows. Without loss of generality, it suffices to show
that every $K_\mathcal{B}^\emptyset$ is the intersection of some $C_I$. We have
\begin{equation*}
K_\mathcal{B}^\emptyset \subseteq \bigcap_{I \subset \mathcal{B}} C_I.
\end{equation*}
On the other hand,
for every facet of $K_\mathcal{B}^\emptyset$, we choose $I$ such that $C_I$ shares
this face and $K_\mathcal{B}^\emptyset$ is contained in $C_I$. This can always be
achieved by choosing $I$ so that every generator of $C_I$ is in the same half space
as $K_\mathcal{B}^\emptyset$. The intersection of these $C_I$ then is contained in
$K^\emptyset_\mathcal{B}$.

Now it remains to show that the intersection of the secondary fans refines the
discriminantal arrangement. This actually follows from the fact, that for every
hyperplane $H_\circuit$, one can choose a minimal generating set which we can
complete to a basis of $A_\Q$ from the $D_i$, where $i \notin \circuit$. By varying
the signs of this generating set, we always get a simplicial cone whose generators
are contained in some secondary fan, and this way $H_\circuit$ is covered by a set of facets
of secondary cones.
\end{proof}

The maximal cones in the secondary fan $\secfan(L, \emptyset)$ correspond to regular
simplicial fans supported on $l_1, \dots, l_n$. More precisely, if $\Delta$ denotes such
a fan, then the corresponding cone is given by $\bigcap_\mathcal{B} K_\mathcal{B}^\emptyset$,
where $\mathcal{B}$ runs over all bases among the $l_i$ which span a maximal cone
in $\Delta$. This definition makes sense for any fan $\Delta$ supported on the $l_i$, we
can single out a specific cone in $\secfan(L, \emptyset)$. Choosing a simplicial model
$\hat{\Delta}$ for $\Delta$, we set:

\begin{definition}
Let $\Delta$ be a fan supported on $L$, then we set:
\begin{equation*}
\nef(\Delta) := \bigcap_{\substack{\mathcal{B} \in \hat{\Delta}\\ \mathcal{B}
\text{ basis in } L}} K_\mathcal{B}^\emptyset
\end{equation*}
and denote $\mathfrak{F}_\nef = \mathfrak{F}_{\nef(\Delta)}$ the discriminantal hull of
$\nef(\Delta)$.
\end{definition}

Note that by our conventions we identify $\mathcal{B} \in \hat{\Delta}$ with the set
of corresponding primitive vectors, or the corresponding rows of $L$, respectively.
Of course, $\nef(\Delta)$ is just the nef cone of the toric variety associated to $\Delta$.

\begin{proposition}\label{realnefgeneral}
We have:
\begin{equation*}
\nef(\Delta) = \bigcap_{\hat{\Delta} \cap (\Delta_\ocircuit)_{\max} \neq \emptyset}
H_\ocircuit.
\end{equation*}
\end{proposition}

\begin{proof}
For some basis $\mathcal{B} \subset \on$, the cone $K_\mathcal{B}^\emptyset$
is simplicial, and for every $i \in \on \setminus \mathcal{B}$, the facet of
$K_\mathcal{B}^\emptyset$ which is spanned by the $D_j$ with $j \notin \mathcal{B}
\cup \{i\}$, spans a hyperplane $H_\circuit$ in $P$. This hyperplane
corresponds to the unique circuit $\circuit \subset \mathcal{B} \cup \{i\}$.
As we have seen before, a maximal cone in $\Delta_\ocircuit$
is of the form $\circuit \setminus \{j\}$ for some $j \in \ocircuit^+$. So we
have immediately:
\begin{equation*}
K_\mathcal{B} = \bigcap_{\substack{\exists F \in (\Delta_\ocircuit)_{\max}\\
\text{with } F \subset \mathcal{B}}} H_\ocircuit
\end{equation*}
and the assertion follows.
\end{proof}

\begin{remark}
If $\Delta = \hat{\Delta}$ is a regular simplicial fan, then
$\nef(\Delta)$ is a maximal cone in the secondary fan.
Let $\ocircuit$ be an oriented circuit such that $\Delta$ is supported on
$\Delta_\ocircuit$ in the sense of \cite{GKZ}, \S 7, Def. 2.9,
and denote $\Delta'$ the fan resulting in the
bistellar operation by changing $\Delta_\ocircuit$ to $\Delta_{-\ocircuit}$.
 Then, by \cite{GKZ}, \S7, Thm. 2.10, the hyperplane
$H_\circuit$ is a proper wall of $\nef(\Delta)$ iff $\Delta'$ is regular, too.
\end{remark}

\subsection{Birational toric geometry, rational divisors,  and vanishing theorems.}\label{birat}

Circuits and their related numerical properties are an important tool in toric
geometry, in particular in the context of the toric minimal model program (see
\cite{Reid83} and \cite{Matsuki}, Chapter 14) and the classification of smooth
toric varieties (see \cite{Oda}, \S 1.6, for instance).
The purpose of this subsection is to clarify the relation of some standard
constructions with the intrinsic circuit geometry of a toric variety. Moreover,
we will give new proof of some standard vanishing theorems from this
point of view. In this
section $\Delta$ denotes a fan associated to a toric variety $X$. $L$ denotes the
row matrix of primitive vectors of rays in $\Delta$. We always assume that
the support of $\Delta$ in $N_\Q$ coincides with the positive span of the $l_i$.
Note that this in particular implies that $\pic(X)$ is torsion free
and $N^1(X) = \pic(X)$.

\paragraph{Some remarks on $\Q$-Cartier divisors on toric varieties}

Recall that a $\Q$-divisor on $X$ is $\Q$-Cartier if an integral multiple is Cartier
in the usual sense. A torus invariant Weil divisor $D = \sum_{i \in \on} c_i D_i$ is
$\Q$-Cartier iff for every $\sigma \in \Delta$ there exists some $m_\sigma \in M_\Q$
such that $c_i = l_i(m)$ for all $i \in \sigma(1)$.

\begin{proposition}\label{QCartierimpliesMCM}
Let $D \in A_{d - 1}(X)$ be a Weil divisor, which is $\Q$-Cartier. Then $\sh{O}(D)$ is
maximally Cohen-Macaulay.
\end{proposition}

\begin{proof}
The MCM property is a local property. So, without loss of generality, it suffices to
consider the restriction of $D$ to some $U_\sigma$. Because $D$ is $\Q$-Cartier,
the hyperplane arrangement $H_i^\uc$ coincides with $H_i^{\underline{0}}$ up to
a translation by $m_\sigma$. Therefore $\mathfrak{s}^\uc(M_\Q)$ coincides with
$\mathfrak{s}^{\underline{0}}(M_\Q)$. Also, all strata of the hyperplane arrangement
$H_i^{\underline{0}}$, except possibly the trivial stratum, are unbounded, and thus
contain lattice points. So, because the structure sheaf  is MCM, it follows that
all simplicial complexes $\hat{\sigma}_I$ are acyclic for $\on \neq I \in \mathfrak{s}^{\underline{0}}
= \mathfrak{s}^\uc$, hence $\sh{O}(D)$ is MCM.
\end{proof}

\begin{remark}
See also \cite{BrunsGubeladze} for another proof of Proposition \ref{QCartierimpliesMCM}.
\end{remark}

The MCM-property is useful, as it allows to replace the $\Ext$-groups by
cohomologies in Serre duality:

\begin{proposition}
\label{dualityprop}
Let $X$ be a normal variety with dualizing sheaf $\omega_X$ and $\sh{F}$ a coherent
sheaf on $X$ such that for every $x \in X$, the stalk $\sh{F}_x$ is MCM over
$\sh{O}_{X, x}$. Then for every $i \in \Z$ there exists an isomorphism
\begin{equation*}
\Ext^i_X\big(\sh{F}, \omega_X\big) \cong H^i\big(X, \shHom(\sh{F}, \omega_X)
\big).
\end{equation*}
\end{proposition}

\begin{proof}
For any two $\sh{O}_X$-modules $\sh{F}, \sh{G}$ there exists the following spectral
sequence
\begin{equation*}
E_2^{pq} = H^p\big(X, \shExt_{\sh{O}_X}^q(\sh{F}, \sh{G})\big) \Rightarrow
\Ext_{\sh{O}_X}^{p + q}(\sh{F}, \sh{G}).
\end{equation*}
We apply this spectral sequence to the case $\sh{G} = \omega_X$. For every closed
point $x \in X$ we have the following identity of stalks:
\begin{equation*}
\shExt^q_{\sh{O}_X} (\sh{F}, \omega_X)_x \cong
\Ext_{\sh{O}_{X, x}}^q(\sh{F}_x, \omega_{X, x}).
\end{equation*}
As $\sh{F}$ is maximal Cohen-Macaulay, the latter vanishes for all $q > 0$, and thus
the sheaf
$\shExt^q_{\sh{O}_X}(\sh{F},$ $\omega_X)$ is the zero sheaf for all
$q > 0$. So the above spectral sequence degenerates and we obtain an isomorphism
\begin{equation*}
H^p(X, \shHom(\sh{F}, \omega_X)) \cong \Ext_X^p(\sh{F}, \omega_X)
\end{equation*}
for every $p \in \Z$.
\end{proof}

In the case where $X$ a toric variety, we have $\omega_X \cong \sh{O}(K_X)$,
where $K_X = -\sum_{i \in \on} D_i$. Then, if $\sh{F} = \sh{O}(D)$ for some $D \in A$,
we can identify $\shHom(\sh{O}(D), \omega_X)$ with $\sh{O}(K_X -D)$:

\begin{corollary}
Let $X$ be a toric variety and $D$ a Weil divisor such that $\sh{O}(D)$ is an MCM
sheaf. Then there is an isomorphism:
\begin{equation*}
\Ext^i_X\big(\sh{O}(D), \omega_X\big) \cong H^i\big(X, \sh{O}_X(K_X - D)\big).
\end{equation*}
\end{corollary}

And by Grothendieck-Serre duality:

\begin{corollary}\label{mcmserreduality}
If $X$ is a complete toric variety and $D$ a Weil divisor such that $\sh{O}(D)$
is an MCM sheaf, then
\begin{equation*}
H^i\big(X, \mathcal{O}(D)\big) \cong H^{d - i}\big(X, \sh{O}(K_X - D)\big)\check{\ }.
\end{equation*}
\end{corollary}

\paragraph{The Picard group.}

The Picard group in a natural way coincides with a flat of the discriminantal
arrangement:

\begin{theorem}[see \cite{Eikelberg}, Theorem 3.2]\label{picardgroup}
Let $X$ be any toric variety, then:
\begin{equation*}
\pic(X) _\Q = \bigcap_{\sigma \in \Delta_{\max}} H_{\sigma(1)} =
\bigcap_{\substack{\circuit \in \circuit(L_{\sigma(1)}) \\ \sigma \in \Delta_{\max}}}
H_\circuit.
\end{equation*}
\end{theorem}

\begin{proof}
As remarked before, a $\Q$-Cartier divisor is specified by a collection $\{m_\sigma\}_{\sigma \in \Delta}
\subset M_\Q$. In particular, all for every $\sigma \in \Delta$, the hyperplanes $H_i^\uc$ with
$i \in \sigma(1)$ have nonempty intersection. So the first equality follows. The second
equality follows from Lemma \ref{intersectionproperty}.
\end{proof}

\paragraph{Nef and Mori cone.}

The nef cone $\nef(X)$ is nothing but $\nef(\Delta)$ as defined in subsection
\ref{secondaryfans}. Let us denote $\mathfrak{F}'_\nef
\subset \mathfrak{F}_\nef$ the subset of those $\ocircuit$ such that $\Delta_\ocircuit \cap
\hat{\Delta} \neq \emptyset$. Moreover, in the case that
$\Delta$ is simplicial, denote $\mathfrak{F}_\nef'' \subset \mathfrak{F}_\nef$ the subset
such that $\Delta$ is supported on $\Delta_\circuit$ in the sense of \cite{GKZ}, \S 7, Def. 2.
Then the following is a consequence of
Proposition \ref{realnefgeneral}:

\begin{theorem}\label{nefcone}
\label{nefconecharacterization}
The nef cone of $X$ is given by the following intersection in $A_\Q$:
\begin{equation*}
\operatorname{nef}(X) = \bigcap_{\ocircuit \in \mathfrak{F}_\nef} H_\ocircuit
= \bigcap_{\ocircuit \in \mathfrak{F}'_\nef} H_\ocircuit.
\end{equation*}
\end{theorem}

Let $\Delta = \hat{\Delta}$ be a simplicial fan supported on $L$.
Then every inner facet $\tau \in
\Delta(d - 1)$, has a canonically  associated circuit.
Namely, $\tau$ is contained in precisely two maximal cones $\sigma,
\sigma' \in \Delta(d)$, and the set $\sigma(1) \cup \sigma'(1)$ contains precisely
$d + 1$ elements, and $\sigma(1)$ as well as $\sigma'(1)$ form a basis of $N_\Q$.
Therefore, the set $\sigma(1) \cup \sigma'(1)$ contains a unique circuit.

\begin{definition}
Let $\tau$ be an inner facet of a simplicial fan $\Delta$. Then we denote $\circuit(\tau)$
its canonically associated circuit.
\end{definition}

Any fan of the form $\Delta_\ocircuit$ for some oriented circuit $\ocircuit$ is simplicial.
We can make use of the calculations of section \ref{onecircuitvarieties} to define linear
forms on $\pic(X)_\Q$. Let $\ocircuit \in \ocircuit(L)$ be any oriented circuit,
then by the short exact sequence
\begin{equation*}
0 \longrightarrow H_\circuit \longrightarrow A_{d - 1}(X)_\Q \longrightarrow A_{\circuit, \Q}
\longrightarrow 0,
\end{equation*}
we can lift the linear forms $t_{\ocircuit, \tau}$, where $\tau$ an inner wall of
$\Delta_\ocircuit$, to linear forms on $A_{d - 1}(X)_\Q$: for every $D \in A_{d - 1}(X)_\Q$
we denote $\bar{D}$
its image in $A_{\circuit, \Q}$ and set $t_{\circuit, \tau}(D) := t_{\circuit,
\tau}(\bar{D})$. It follows easily that any $t_{\ocircuit, \tau}$
and $t_{\ocircuit, \tau'}$ are proportional to each other by the factors
$\frac{s_\tau}{s_{\tau'}} \frac{\lcm\{\alpha_i, \alpha_j\}}{\lcm\{\alpha_p,
\alpha_q\}}$, where $\{i, j\} = \circuit \setminus \tau(1)$ and $\{p, q\} = \circuit
\setminus \tau'(1)$ and the defining relation for $\ocircuit$ is $\sum_{i \in \circuit}
\alpha_i l_i = 0$.

\begin{remark}
Note that in the case that $X$ is nonsimplicial, the form $t_{\ocircuit, \tau}$ should
not be identified with an actual curve class. We only have a linear form on $A_{d - 1}(X)_\Q$.
We can think of the $t_{\ocircuit, \tau}$ as 'virtual' curve classes on $X$.
Also note that the facet $\tau$ in general is not realized as a facet in $\Delta$.
\end{remark}

If $X$ is complete, by the nondegenerate pairing
\begin{equation*}
N_1(X)_\Q \otimes_\Q \pic(X)_\Q \longrightarrow \Q
\end{equation*}
we can identify $N_1(X)_\Q$ with a quotient vector space of the dual vector space
$A_{d - 1}(X)\check{\ }_\Q$ and in fact the $t_{\ocircuit, \tau}$ become rational
equivalence classes of curves after projection to $N_1(X)_\Q$. Denote these projections
$\bar{t}_{\ocircuit, \tau}$, then we have by the duality between the nef cone
and the Mori cone:

\begin{theorem}\label{moriconetheorem}
Let $X$ be a complete toric variety, then the Mori cone of $X$ in $N_1(X)_\Q$ is
generated by the $\bar{t}_{\ocircuit, \tau}$, where $\ocircuit \in \mathfrak{F}'_\nef$.
\end{theorem}

\begin{remark}
Classes of circuits on which the fan $\Delta$ is supported have been considered earlier in
\cite{Casagrande03} and were called ``contractible classes''.
A contractible class is extremal iff
the result of the associated flip is a projective toric variety again. See also
\cite{Bonavero00} for examples. If $X$ is simplicial and
projective, the Mori cone is a strictly convex
polyhedral cone in $N_1(X) = A_{d - 1}(X)\check{\ }_\Q$ and by the theorem, $\mathfrak{F}'_\nef$
contains the set of its extremal rays. The Mori cone is generated by the $t_{\ocircuit,
\tau}$ such that $V(\tau)$ is an extremal curve. However, in general, $\mathfrak{F}'_\nef$
is strictly bigger than the set of extremal curve classes.
 \end{remark}

\paragraph{The Iitaka dimension of a nef divisor and Kawamata-Viehweg vanishing.}

Let $D$ be a Cartier divisor on some normal variety $X$, and denote $N(X, D) :=
\{k \in \N \mid H^0\big(X, \sh{O}(kD)\big) \neq 0\}$. Then the Iitaka dimension
of $D$ is defined as
\begin{equation*}
\kappa(D) := \max_{k \in N(X, D)}\{\dim \phi_k(X)\},
\end{equation*}
where $\phi_k : \xymatrix{X \ar@{-->}[r] & \mathbb{P} |kD|}$
is the family of morphisms given by the linear series $|kD|$.

In the case where $X$ is a toric variety and $D = \sum_{i \in \on} c_i D_i$ invariant,
the Iitaka dimension of $D$ is just the dimension of $P_{kD}$ for $k >> 0$.
For a $\Q$-Cartier divisor $D$, we define its Iitaka dimension by $\kappa(D) :=
\kappa(rD)$ for $r > 0$ such that $rD$ is Cartier.

If $D$ is a nef divisor, then the morphism $\phi : X \longrightarrow \mathbb{P}
|D|$ is torus equivariant, its image is a projective toric variety of dimension $\kappa(D)$
whose associated fan is the inner normal fan of $P_D$. If $\kappa(D) < d$,
then necessarily $D$ is contained in some hyperplane $H_\circuit$ such that
$\ocircuit^+ = \circuit$ for some orientation $\ocircuit$ of $\circuit$. The toric
variety associated to $\ocircuit$ is isomorphic to a finite cover of a weighted
projective space. This kind of circuit will play an important role later on, so that
we will give it a distinguished name:

\begin{definition}
We call a circuit $\circuit$ such that $\circuit = \ocircuit^+$ for one of its orientations,
{\em fibrational}.
\end{definition}

By Proposition
\ref{orthantboundary}, this implies that $D$ is contained in the intersection of
$\nef(X)$ with the effective cone of $X$, which we identify with $C_\emptyset$.
More precisely, it follows from linear algebra that $D$ is contained in all $H_\circuit$
where $\circuit$ is fibrational and $l_i(P_D) = 0$ for all $i \in \circuit$.

\begin{definition}
Let $D \in A_{d - 1}(X)_\Q$, then we denote $\operatorname{fib}(D) \subset \circuit(L)$
the set of fibrational circuits such that $D \in H_\circuit$.
\end{definition}

The fibrational circuits of a nef divisor $D$ tell us immediately about its Iitaka dimension:

\begin{proposition}
Let $D$ be a nef $\Q$-Cartier divisor. Then $\kappa(D) = d - \rk L_{T}$, where
$T := \bigcup_{\circuit \in \operatorname{fib}(D)} \circuit$.
\end{proposition}

\begin{proof}
We just remark that $\rk L_{T}$ is the dimension of the subvector space of $M_\Q$
which is generated by the $l_i$ which are contained in a fibrational circuit.
\end{proof}

Recall from section \ref{onecircuitgeneralcohomologyvanishing} that for a toric
$1$-circuit variety $\mathbb{P}(\underline{\alpha}, \xi)$, cohomology vanishing
is determined by the set $F_\ocircuit \subset A_{d - 1}\big(\mathbb{P}(\underline{\alpha}, \xi)\big)$.
The image of $F_\ocircuit$ in $A_{d - 1}\big(\mathbb{P}(\underline{\alpha}, \xi)\big)_\Q \cong \Q$
contains all classes $D > K_\ocircuit$, where $K_\ocircuit = -\sum_{i \in \ocircuit^+} D_i$ (see
section \ref{onecircuitnefkawamataviehweg}), i.e. all classes which are contained in the
open interval $(K_\ocircuit, \infty)$.
In particular, in the case where $\circuit$ is {\em not} fibrational, it also contains the canonical divisor
$K_{\mathbb{P}(\underline{\alpha}, \xi)} = -\sum_{i \in \circuit} D_i$.

\begin{proposition}\label{antinefcohomology}
Let $X$ be a complete toric variety and $D$ a nef divisor, then
$H^i\big(X, \sh{O}(-D)\big) = 0$ for $i \neq \kappa(D)$.
\end{proposition}

\begin{proof}
Consider the hyperplane arrangement given by the $H_i^\uc$ in $M_\Q$. Let $m \in M_\Q$
and $I = \mathfrak{s}^\uc(m)$. Then the simplicial complex $\hat{\Delta}_I$ can be
characterized as follows. Consider $Q \subset P_D$ the union of the set faces of $P_D$
which are contained in any $H_i^\uc$ with $i \in I$. This is precisely the portion of $P_D$,
which the the point $m$ ``sees'', and therefore contractible, where the convex hull of $Q$
and $m$ provides the homotopy between $Q$ and $m$. Therefore, every $\hat{\Delta}_I$
is contractible with an exception for $I = \emptyset$, because $\hat{\Delta}_\emptyset
= \emptyset$, which is not acyclic with respect to reduced cohomology. Now we pass to
the inverse, i.e. we consider the signature of $-m$ with respect to $H_i^{-\uc}$. Then for
any such $-m$ which does not sit in the relative interior of the polytope $P^\on_{-\uc}$,
there exists $m' \in M_\Q$ with signature $\mathfrak{s}^\uc(m') =: J$ such that
$\hat{\Delta}_J$ is contractible and $\mathfrak{s}^{-\uc}(m) = \on \setminus J$.
As $\hat{\Delta}$ is homotopic to a ${d - 1}$-sphere, we can apply Alexander duality
and thus the simplicial complex $\hat{\Delta}_{\on \setminus J}$ is acyclic.
Thus there remain only the elements in the relative interior of $P^\on_{-\uc}$.
Let $m$ be such an element with signature $I$, then $\hat{\Delta}_I$ is isomorphic
to a $d - \kappa(D) - 1$-sphere, and the assertion follows.
\end{proof}

This proposition implies the toric Kodaira and Kawamata-Viehweg vanishing
theorems (see also \cite{Mustata1}):

\begin{theorem}[Kodaira \& Kawamata-Viehweg]\label{KodairaKawamataViehweg}
Let $X$ be a complete toric variety and $D$, $E$ $\Q$-divisors with $D$ nef and
$E = \sum_{i \in \on} e_i D_i$ with $-1 < e_i < 0$ for all $i \in \on$. Then:
\begin{enumerate}[(i)]
\item\label{kodaira} if $D$ is integral, then $H^i\big(X, \sh{O}(D + K_X)\big) = 0$ for all
$i \neq 0, d - \kappa(D)$;
\item\label{kawamataviehweg} if $D + E$ is a Weil divisor, then
$H^i\big(X, \sh{O}(D + E)\big) = 0$ for all $i > 0$.
\end{enumerate}
\end{theorem}

\begin{proof}
By Proposition \ref{QCartierimpliesMCM} we can apply Serre duality (Corollary \ref{mcmserreduality})
and obtain $H^i\big(X, \sh{O}(D + K_X)\big) \cong H^{d - i}\big(X, \sh{O}(-D\big)$ and
(\ref{kodaira}) follows from Proposition \ref{antinefcohomology}.

For (\ref{kawamataviehweg}):  $D + E$ is contained the interior
of every half space $K_X + H_\ocircuit$ for $\ocircuit \in \mathfrak{F}_\nef$,
and the result follows.
\end{proof}

\section{Arithmetic aspects of cohomology vanishing}\label{arithmeticaspects}

In this section we are going to combine aspects from sections \ref{onecircuitvarieties} and
\ref{discriminantalarrangementsandsecfans}. In particular, we want to derive vanishing
results for integral divisors which cannot directly be derived from the setting of $\Q$-divisors.
From now on we assume that the $l_i$ are integral.
Recall that for any integral divisor $D = \sum_{i \in \on} c_i D_i$ and any torus invariant closed
subvariety $V$ of $X$, vanishing of $H_V^i\big(X, \sh{O}(D)\big)$ depends on two things:
\begin{enumerate}[(i)]
\item whether the set of signatures $\mathfrak{s}^\uc(M_\Q)$ consists of $I \subset \on$
such that the relative cohomology groups $H^{i - 1}(\hat{\Delta}_I, \hat{\Delta}_{V, I}; k)$ vanish, and,
\item if $H^{i - 1}(\hat{\Delta}_I, \hat{\Delta}_{V, I}; k)$ for one such $I$, whether the
corresponding polytope $P_\uc^I$ contain lattice points $m$ with $\mathfrak{s}^\uc(m) = I$.
\end{enumerate}
In the Gale dual picture, the signature $\mathfrak{s}^\uc(M_\Q)$ coincides with the set of $I \subset \on$
such that the class of $D$ in $A_{d - 1}(X)_\Q$ is contained in $C_I$. For fixed $I$, the classes
of divisors $D$  in $A_{d - 1}(X)$ such that the equation $l_i(m) < -c_i$ for $i \in I$ and
$l_i(m) \geq -c_i$ for $i \notin I$ is satisfied, is counted by the {\em generalized partition function}.
That is, by the function
\begin{equation*}
D \mapsto \Big| \{ (k_1, \dots, k_n) \in \N^n \mid \sum_{i \in \on \setminus I} k_i D_i - \sum_{i \in I} k_i D_i = D
\text{ where } k_i > 0 \text{ for } i \in I\} \Big|.
\end{equation*}
So, in the most general picture, we are looking for $D$ lying in the common zero set of the
vector partition function for all relevant signatures $I$ of $D$. In general, this is a difficult
problem to determine these zero sets,
and it is hardly necessary for practical purposes.

\begin{remark}\label{frobeniusremark}
The diophantine Frobenius problem (also known as money change problem or
denumerant problem) is a classical problem of number theory and so far it is
unsolved, though for fixed $n$, there exist polynomial time algorithms to
determine the zeros of  the vector partition function. Though there do exist
explicit formulas for the Ehrhart quasipolynomials (see for instance \cite{ChenTuraev}),
a general closed solution is only known
for the case $n = 2$. For a general overview we refer to the book
\cite{RamirezAlfonsin}. Vector partition functions play an
important role in the combinatorial theory of rational polytopes and have been
considered, e.g. in \cite{Sturmfels95}, \cite{BrionVergne97} (see also references
therein). In \cite{BrionVergne97} closed expressions in terms of residue formulas
have been obtained. Moreover it was shown that the vector partition function
is a piecewise
quasipolynomial function, where the domains of quasipolynomiality are
chambers (or possibly unions of chambers) of the secondary fan. In particular, for
if $P_\uc^\emptyset$ is a rational bounded polytopy, then the values of the
vector partition function for $P_{k \cdot \uc}^\emptyset$ for
$k \geq 0$, is just the Ehrhart quasipolynomial.
\end{remark}

 A first
--- trivial --- approximation is given by the observation that the divisors $D$ where the vector
partition function takes a nontrivial value map to the cone $C_I$, shifted by the {\em offset}
$e_I := -\sum_{i \in I} e_i$. This offset is essentially the same as the offset $K_\ocircuit$ of
section \ref{onecircuitvarieties}.

\begin{definition}
We denote $\mathbb{O}'(L, I)$ the saturation of  cone generated the $-D_i$ for $i \in I$ and the
$D_i$ for $i \notin I$ and $\mathbb{O}(L, I) := e_I + \mathbb{O}'(L, I)$. Moreover, we denote
$\Omega(L, I)$ the zero set in $\mathbb{O}(L, I)$ of the vector partition function as defined above.
\end{definition}

In the next step we want to approximate the sets $\Omega(L, I)$ by reducing to the classical
diophantine Frobenius problem. For this, fix some $I \subset \on$ and consider some polytope
$P_\uc^I$. It follows from Proposition \ref{orthantboundary} that $D$ is contained in the
intersection of half spaces $H_\ocircuit$ for $\ocircuit \in \ocircuit(L)$ such that $\ocircuit^- =
\circuit \cap I$. In the polytope picture, we can interpret this as follows. For every $\ocircuit$
and its underlying circuit $\circuit$, we set
\begin{equation*}
P^\ocircuit_\uc := \{m \in M_\Q \mid l_i(m) \leq -c_i \text{ for } i \in \ocircuit^- \text{ and }
l_i(m) \geq -c_i \text{ for } i \in \ocircuit^+\}.
\end{equation*}
Consequently, we get
\begin{equation*}
P_\uc^I = \bigcap_{\ocircuit} P^\ocircuit_\uc,
\end{equation*}
where the intersection runs over all $\ocircuit \in \ocircuit(L)$ with $\ocircuit^- = \circuit \cap I$.
It follows that if there exists a compatible oriented circuit $\ocircuit$ such that  $P_\uc^\ocircuit$
does not contain a lattice point, then $P_\uc^I$ also does not contain a lattice point. We want
to capture this by considering an arithmetic analogue of the discriminantal arrangement in
$A_{d - 1}(X)$ rather than in $A_{d - 1}(X)_\Q$. Using the integral pendant to diagram
(\ref{HIdefdiagram})  (compare definition \ref{fcdef}):

\begin{definition}
Consider the morphism $\eta_I: A_{d - 1}(X) \twoheadrightarrow A_I$. Then we denote
$Z_I$ its kernel. For $I = \circuit$ and $\ocircuit$ some orientation of $\circuit$ we denote
by $F_\ocircuit$ the preimage in $A_{d - 1}(X)$ of the complement
of the semigroup consisting of elements $\sum_{i \in \ocircuit^-} c_i D_i - \sum_{i \in \ocircuit^+} c_i D_i $,
where $c_i \geq 0$ for $i \in \ocircuit^-$ and $c_i > 0$ for $i \in \ocircuit^+$.
We set $F_\circuit := F_\ocircuit \cap F_{-\ocircuit}$.
\end{definition}

So, there are two candidates for a discriminantal arrangement in $A_{d - 1}(X)$, the $Z_\circuit$
on the one hand, and the $F_\circuit$ on the other.

\begin{definition}
We denote:
\begin{itemize}
\item $\{Z_\circuit\}_{\circuit \in \circuit(L)}$ the {\em integral} discriminantal arrangement, and
\item $\{F_\circuit\}_{\circuit \in \circuit(L)}$ the {\em Frobenius} discriminantal arrangement.
\end{itemize}

\end{definition}

The integral discriminantal arrangement has similar properties as the
$H_I$, as they give a solution to the integral discriminant problem (compare Lemma \ref{intersectionproperty}):

\begin{lemma}\label{zintersectionproperty}
Let $I \subset \on$, then
\begin{equation*}
Z_I = \bigcap_{\circuit \in \circuit(L_I)} Z_\circuit.
\end{equation*}
\end{lemma}

As in the rational case, we can use this to locate the integral Picard group in
$A_{d - 1}(X)$:

\begin{theorem}[see \cite{Eikelberg}, Theorem 3.2]
Let $X$ be any toric variety, then:
\begin{equation*}
\pic(X) = \bigcap_{\sigma \in \Delta_{\max}} Z_{\sigma(1)} =
\bigcap_{\substack{\circuit \in \circuit(L_{\sigma(1)}) \\ \sigma \in \Delta_{\max}}}
Z_\circuit.
\end{equation*}
\end{theorem}

\begin{proof}
A Cartier divisor is specified by a collection $\{m_\sigma\}_{\sigma \in \Delta} \subset M$
such that the hyperplanes $H_i^\circuit$ with $i \in \sigma(1)$ intersect in
integral points. So the first equality follows.
The second equality follows from Lemma \ref{zintersectionproperty}.
\end{proof}

The Frobenius discriminantal arrangement is not as straightforward. First, we note
the following properties:

\begin{lemma}\label{movedawaylemma}
Let $\circuit \in \circuit(L)$, then:
\begin{enumerate}[(i)]
\item\label{movedawaylemmai} $F_\circuit$ is nonempty;
\item\label{movedawaylemmaii}  the saturation of $Z_\circuit$ in $A_{d - 1}(X)$ is contained in
$F_\circuit$ iff $\circuit$ is not fibrational.
\end{enumerate}
\end{lemma}

\begin{proof}
The first assertion follows because $F_\circuit$ contains all elements which map
to the open interval $(K_\ocircuit, K_{-\ocircuit})$ in $A_{\circuit, \Q}$.
For the second assertion, note that the set $\{m \in M \mid l_i(m) = 0 \text{ for all }
i \in \circuit\}$ is in $F_\circuit$ iff $\ocircuit^+ \neq \circuit$ for either orientation
$\ocircuit$ of $\circuit$.
\end{proof}

Lemma \ref{movedawaylemma} shows that the $F_\circuit$ are thickenings of
the $Z_\circuit$ with one notable exception, where $\circuit$ is fibrational.
In this case, $F_\circuit$ can be considered as parallel to, but slightly shifted
away from $Z_\circuit$. In the sequel we will not make any explicit use of the
$Z_\circuit$ anymore, but these facts should be kept in mind.

Regarding the Frobenius discriminantal arrangement, we want also to consider
integral versions of the discriminantal strata:

\begin{definition}\label{vanishingcoredefinition}
Let $\ocircuit \in \ocircuit(L)$ and let $\mathfrak{F}_\mathcal{S}$ be a
discriminantal hull of $\mathcal{S} = \{S_1, \dots, S_k\}$, then we denote
\begin{equation*}
\mathfrak{A}_\mathcal{S} := \bigcap_{\ocircuit \in \mathfrak{F}_\mathcal{S}} F_\ocircuit.
\end{equation*}
the {\em arithmetic  core} of $\mathfrak{F}_\mathcal{S}$.
In the special case $\mathfrak{F}_\mathcal{S} = \mathfrak{F}_\nef$ we write
$\mathfrak{A}_\nef$.
\end{definition}

\begin{remark}
The notion {\em core} refers to the fact that we consider {\em all} $F_\ocircuit$, instead
of a non-redundant subset describing the set $S$ as a convex cone.
\end{remark}

We will use arithmetic cores to derive arithmetic versions of known vanishing theorems
formulated in the setting of $\Q$-divisors and to get refined conditions on cohomology
vanishing. This principle is reflected in the following theorem:

\begin{theorem}\label{integralvanishingtheorem}
Let $V$ be a $T$-invariant closed subscheme of $X$ and $S$ a discriminantal
stratum in $A_{d - 1}(X)_\Q$. If $H^i_V\big(X, \sh{O}(D)\big) = 0$ for some $i$
and for all integral divisors $D \in S$, then also $H^i_V\big(X, \sh{O}(D)\big) = 0$
for all $D \in \mathfrak{A}_S$.
\end{theorem}

\begin{proof}
Without loss of generality we can assume that $\dim S > 0$.
Consider some nonempty $P_\uc^I$ for some $I \subset \on$. Then for any
such $I$, we can choose some multiple of $k D$ such that $P_{k\uc}^I$
contains a lattice point. But if $H^i_V\big(X, \sh{O}(D)\big) = 0$, then
also $H^i_V\big(X, \sh{O}(kD)\big) = 0$, hence $H^{i - 1}(\hat{\Delta}_I,
\hat{\Delta}_{V, I}; k) = 0$. Now, any divisor $D' \in \mathfrak{A}_S$
which does not map to $S$, is contained in $F_{\ocircuit}$ for all
$\ocircuit \in \mathfrak{F}_S$ and therefore for any $I$ which is in
the signature for $D'$ but not for $D$, the equations $l_i(m) < -c_i'$
for $i \in I$ and $l_i(m) \geq -c_i'$ for $i \notin I$ cannot have any
integral solution.
\end{proof}

We apply Theorem  \ref{integralvanishingtheorem} to $\mathfrak{A}_\nef$:

\begin{theorem}[Arithmetic version of Kawamata-Viehweg vanishing]\label{nefvanishing}
Let $X$ be a complete toric variety. Then $H^i\big(X, \sh{O}(D)\big) = 0$
for all $i > 0$ and all $D \in \mathfrak{A}_\nef$.
\end{theorem}

\begin{proof}
We know that the assertion is true if $D$ is nef. Therefore we can apply Theorem
\ref{integralvanishingtheorem} to the maximal strata $S_1, \dots, S_k$ of $\nef(X)$.
Therefore the assertion  is true for $D \in \bigcap_{i = 1^k} \mathfrak{A}_{S_i}$. To
prove the theorem, we have to get rid of the $F_\ocircuit$, where $H_\circuit$ intersects
the relative interior of a face of $\nef(X)$. Let $\circuit$ be such a circuit and $R$ the
face. Without loss of generality, $\dim R > 0$. Then we can choose elements $D'$ in
$R$ at an arbitrary distance from $H_\ocircuit$, i.e. such that the polytope $P^\ocircuit_\uc$
becomes arbitrarily big and finally contains a lattice point. Now, if we move outside
$\nef(X)$, but stay inside $\mathfrak{A}_\nef$, the lattice points of $P^\ocircuit_\uc$
cannot acquire any cohomology and the assertion follows.
\end{proof}

One can imagine an analog of the set $\mathfrak{A}_S$ in $A_{d - 1}(X)_\Q$
as the intersection of shifted half spaces
\begin{equation*}
\bigcap_{\ocircuit \in \mathfrak{F}_S} K_\ocircuit + H_\ocircuit.
\end{equation*}
The main difference here is that one would picture the proper facets of this convex
polyhedral set as ``smooth'', whereas the proper ``walls'' of $\mathfrak{A}_S$
have ``ripples'', which arise both from the fact that the groups $A_\circuit$ may have torsion,
and that we use Frobenius conditions to determine the augmentations of our half spaces.

In general, the set $\mathfrak{F}_S$ is highly redundant, when it comes
to determine $\overline{S}$, which implies that above intersection does not yield
a cone but rather a polyhedron, whose recession cone corresponds to $\overline{S}$.
In the integral situation we do not quite have a recession cone, but a similar
property holds:

\begin{proposition}\label{asymptoticproperty}
Let $V \subset X$ be a closed invariant subscheme and $\mathcal{S} = \{S_1, \dots, S_k\}$ a
collection of discriminantal stata different from zero such that
$H^i_V\big(X, \sh{O}(D)\big) = 0$ for $D \in \mathfrak{A}_S$. Then for any nonzero face
of its discriminantal hull $\overline{S}$ there exists the class
of an integral divisor $D' \in \overline{S}$ such that the intersection of the half line $D + r D'$
for $0 \leq r \in \Q$ with $\mathfrak{A}_\mathcal{S}$ contains infinitely many classes of integral divisors.
\end{proposition}

\begin{proof}
Let $R \subset \overline{S}$ be any face of $S$, then the vector space spanned by $R$ is
given by an intersection $\bigcap_{\circuit \text{ with } \ocircuit \in K} H_\circuit$ for a certain
subset $K \subset \mathfrak{F}_\mathcal{S}$. We assume that $K$ is maximal with this property. The intersection
$\bigcap_{\ocircuit \in K} F_\ocircuit$ is invariant with respect to translations along certain
(though not necessarily all) $D' \in R$. This implies that the line (or any half line, respectively),
generated by $D'$ intersects $\bigcap_{\ocircuit \in K} F_\ocircuit$ in infinitely many points.
As $K$ is maximal, there is no other $\ocircuit \in \mathfrak{F}_\circuit$ parallel to $R$ and
the assertion follows.
\end{proof}

The property of Proposition \ref{asymptoticproperty} is necessary for elements in $\mathfrak{A}_\mathcal{S}$,
but not sufficient. This leads to the following definition:

\begin{definition}
Let $\mathcal{S} = \{S_1, \dots, S_k\}$ be a collection of nonzero discriminantal strata and
$D \in A_{d - 1}(X)$ such that the property of
Proposition \ref{asymptoticproperty} holds. If $S$ is not contained in $\mathfrak{A}_\mathcal{S}$,
then we call $D$ $\mathfrak{A}_S$-residual. If $S = 0$, then we write
$0$-residual instead of $\mathfrak{A}_0$-residual.
\end{definition}

Note that, by definition, every divisor outside $\bigcap_{\ocircuit \in \ocircuit(L)} F_\ocircuit$ is
$0$-residual.

In the next subsections we will consider several special cases of interest for cohomology vanishing,
which are not directly related to Kawamata-Viehweg vanishing theorems. In subsection \ref{nonstandardvanishing}
we will consider global cohomology vanishing for divisors in the inverse nef cone. In subsection
\ref{nonstandardsurfacevanishing} we will present a more explicit determination of this type of
cohomology vanishing for toric surfaces. Finally, in subsection \ref{mcmsection}, we will give
a geometric criterion for determing maximally Cohen-Macaulay modules.

\subsection{Nonstandard Cohomology Vanishing}\label{nonstandardvanishing}

In this subsection we want to give a qualitative description of cohomology vanishing which
is related to divisors which are {\em inverse} to nef divisors of Iitaka dimension $0 < \kappa(D)
< d$. 
We show the following theorem:

\begin{theorem}\label{projectiveantinef}
Let $X$ be a complete $d$-dimensional toric variety. Then $H^i\big(X, \sh{O}(D)\big) = 0$
for every $i$ and all $D$ which are contained in some $\mathfrak{A}_{-F}$,
where $F$ is a face of  $\nef(X)$ which
cointains nef divisors of Iitaka dimension
$0 < \kappa(D) < d$. If $\mathfrak{A}_{-F}$ is nonempty, then it contains infinitely many
divisor classes.
\end{theorem}

\begin{proof}
Recall that such a divisor, as a $\Q$-divisor,  is contained in the intersection
$\bigcap_{\circuit \in \operatorname{fib}(D)} H_\circuit$ and therefore it is in the intersection
of the nef cone with the boundary of the effective cone of $X$ by Proposition \ref{orthantboundary}.
Denote this intersection by $F$. Then we claim that $H^i\big(X, \sh{O}(D')\big) = 0$ for
all $D' \in \mathfrak{A}_{-F}$. By Corollary \ref{antinefcohomology} we know that
$H^i\big(X, \sh{O}(E)\big) = 0$ for $0 \leq i < d$ for any divisor $E$ in the interior of
the inverse nef cone. This implies that $H^i\big(X, \sh{O}(E)\big) = 0$ for any $E
\in \mathfrak{A}_{-nef}$ and hence $H^i\big(X, \sh{O}(D')\big) = 0$ for any $D'
\in \mathfrak{A}_{-F}$, because $\mathfrak{A}_{-F}  \subset \mathfrak{A}_{-\nef}$.
The latter assertion follows from the fact that the assumption on the Iitaka
dimension implies that the face $F$ has positive dimension.
\end{proof}

Note that criterion is not very strong, as it is not clear in general whether
the set $\mathfrak{A}_{-F}$ is nonempty. However, this is the case in a
few interesting cases, in particular for toric surfaces, as we will see in the
next subsection. The following remark shows that our condition indeed is
rather weak in general:

\begin{remark}\label{nefremark}
The inverse of any big and nef divisor $D$ with the property that $P_D$ does not contain
any lattice point in its interior has the property that $H^i\big(X, \sh{O}(D)\big) = 0$
for all $i$. This follows directly from the standard fact in toric geometry that the
Euler characteristics $\chi(-D)$ counts the inner lattice points of the lattice polytope
$P_D$.
\end{remark}

\subsection{The case of complete toric surfaces.}\label{nonstandardsurfacevanishing}

Let $X$ be a complete toric surface. We assume that the $l_i$ are circularly ordered.
We consider the integers $\on$ as system of representatives for $\Z / n \Z$,
i.e. for some $i \in \on$ and $k \in \Z$, the sum $i + k$ denotes the unique element
in $\on$ modulo $n$.

\begin{proposition}\label{surfacenefstratum}
Let $X$ be a complete toric surface. Then $\nef(X) = \overline{S}$, where $S$ is
a single stratum of maximal dimension of the discriminantal arrangement.
\end{proposition}

\begin{proof}
$X$ is simplicial and projective and therefore $\nef(X)$ is a cone of maximal dimension
in $A_1(X)_\Q$.
We show that no hyperplane $H_\circuit$ intersects the interior of $\nef(X)$.
By Proposition \ref{orthantboundary} we can at once exclude fibrational circuits.
This leaves us with non-fibrational circuits
$\circuit$ with cardinality three, having orientation $\ocircuit$ with $|\ocircuit^+| = 2$.
Assume that $D$ is contained in the interior of $H_{-\ocircuit}$. Then there exists
$m \in M_\Q$ such that $\ocircuit^+ \subset \mathfrak{s}^D(m)$, which implies that
the hyperplane $H_i^\uc$ for $\{i\} = \ocircuit^-$ does not intersect $P_D$, and thus $D$
 cannot be nef. It follows that $\nef(X) \subset H_\ocircuit$.
\end{proof}

Now assume there exist $p, q \in \on$ such that $l_q = -l_p$, i.e. $l_p$ and $l_q$
represent a one-dimensional fibrational circuit of $L$. Then for any nef divisors $D$
which is contained in $H_{p, q}$, the associated polytope $P_D$ is one-dimensional.
The only possible variation for $P_D$ is its length in terms of lattice units. So we can
conclude that $\nef(X) \cap H_{p, q}$ is a one-dimensional face of $\nef(X)$.

\begin{definition}
Let $X$ be a complete toric surface and $\circuit = \{p, q\}$ such that $l_p = -l_q$.
Then we denote $S_{p, q}$ the relative interior of $-\nef(X) \cap H_\circuit$.
Moreover, we denote $\mathfrak{A}_{p, q}$ the arithmetic core of $S_{p, q}$.
\end{definition}

Our aim in this subsection is to prove the following:

\begin{theorem}
\label{surfaceclassification}
Let $X$ be a complete toric surface. Then there are only finitely many divisors
$D$ with $H^i\big(X, \sh{O}(D)\big) = 0$ for all $i > 0$ which are not contained
in $\mathfrak{A}_\nef \cup \bigcup \mathfrak{A}_{p, q}$, where the union
ranges over all pairs $\{p, q\}$ such that $l_p = -l_q$.
\end{theorem}

We will prove this theorem in several steps. First we show that the interiors of the
$C_I$ such that $H^0(\hat{\Delta}_I; k) \neq 0$ cover all of $A_1(X)_\Q$ except
$\nef(X)$ and $-\nef(X)$.

\begin{proposition}\label{surfacecores}
Let $D = \sum_{i \in \on} c_i D_i$ be a Weil divisor which is not contained in
$\nef(X)$ or $-\nef(X)$, then the corresponding arrangement $H^\uc_i$ in
$M_\Q$ has a two-dimensional chamber $P_\uc^I$ such that
complex $\hat{\Delta}_I$ has at least two components.
\end{proposition}

\begin{proof}
Recall that $\nef(X) = \bigcap H_\ocircuit$, where the intersection runs over all
oriented circuits which are associated to extremal curves of $X$. As the statement
is well-known for the case where $X$ is either a $1$-circuit toric variety or a
Hirzebruch surface, we can assume without loss of generality,
that the extremal curves belong to blow-downs, i.e. the associated oriented
circuits are of the form $\ocircuit^+ = \{i - 1, i + 1\}$, $\ocircuit^- = \{i\}$
for any $i \in \on$. Now assume that $D$ is in the interior of $H_\ocircuit$
for such an oriented
circuit $\ocircuit$. Then there exists a bounded chamber $P_\uc^I$ in $M_\Q$
such that $\ocircuit^- = \circuit \cap \mathfrak{s}^\uc(m)$.
In order for $\hat{\Delta}_{\mathfrak{s}^\uc(m)}$ to be acyclic, it is
necessary that $\mathfrak{s}^\uc(m) \cap (\on \setminus \circuit) = \emptyset$.
Let $\{j, k, l\} =: \mathcal{D} \subset \on$ represent any other circuit such that
$\mathfrak{D}^+ = \{j, l\}$ for some orientation $\mathfrak{D}$ of $\mathcal{D}$.
The hyperplane arrangement given by the three hyperplanes $H^{\underline{c}}_j,
H^{\underline{c}}_k, H^{\underline{c}}_l$ has six unbounded regions, whose
signatures contain any subset of $\{j, k, l\}$ except $\{j, l\}$ and $\{k\}$.
In the cases $j = i - 2, k = i - 1, l = 1$ or $j = i, k = i + 1, l = i + 2$, $P_\uc^I$
must be contained in the region with signature $\{i\}$.
In every other case $P_\uc^I$ must be contained in the region with signature
$\emptyset$.
In the case, say, $\{j, k, l\} = \{i - 2, i - 1, i\}$, the hyperplane
$H^{\underline{c}}_{i - 2}$ should not cross the bounded chamber related
to the subarrangement given by the hyperplanes $H_{i -1}^\uc$, $H_i^\uc$,
$H_{i + 1}^\uc$, as else we obtain a chamber whose signature contains
$\{i - 1, i + 1\}$, but not $\{i - 2, i\}$. Then the associated subcomplex of
$\hat{\Delta}$ can never be acyclic.
This implies that, if $D$ is in the interior of  $H_\ocircuit$, then
$D \in H_{\mathfrak{D}}$, where either $\mathcal{D} = \{i - 2, i - 1, i\}$
or $\mathcal{D} = \{i, i + 1, i + 2\}$.
By iterating for every extremal (i.e. every invariant) curve, we conclude that
$D \in \bigcap_{i \in \on} H_\ocircuit = \nef(X)$. Analogously, we conclude
for $D \in H_{-\ocircuit}$ that $D \in -\nef(X)$, and the statement follows.
\end{proof}

Let $\{p, q\} \subset \on$ such that $l_p = - l_q$. Then these two primitive
vectors span a $1$-dimensional subvector space of $N_\Q$, which naturally
separates the set $\on \setminus \{p, q\}$ into two subsets.

\begin{definition}
Let $\{p, q\} \subset \on$ such that $l_p = - l_q$. Then we denote $A_{p, q}^1, A_{p, q}^2
\subset \on$ the two subsets of $\on \setminus \{p, q\}$ separated by the line
spanned by $l_p, l_q$.
\end{definition}

For some fibrational circuit $\{p, q\}$, the
closure $\overline{S}_{p, q}$ is a one-dimensional
cone in $A_1(X)_\Q$ which has a unique primitive vector:

\begin{definition}
Consider$\{p, q\}$ as before. Then the closure $\overline{S}_{p, q}$ is a one-dimensional
cone with primitive lattice vector $D_{p, q} := \sum_{i \in A_{p, q}^1} l_i(m) D_i$,
where $m \in M$ the unique
primitive vector on the ray in $M_\Q$ with $l_p(m) = l_q(m) = 0$ and
$l_i(m) < 0$ for $i \in A_{p, q}^1$.
\end{definition}

\begin{proposition}\label{nosurfaceantinefresiduals}
Let $X$ be a complete toric surface. Then every $\mathfrak{A}_{p, q}$-residual divisor
on $X$ is either contained in $\mathfrak{A}_\nef$, or in some $\mathfrak{A}_{p, q}$, or
is $\mathfrak{A}_\nef$-residual.
\end{proposition}

\begin{proof}
For any nef divisor $D \in -S_{p, q}$, the polytope $P_D$ is
a line segment such that all $H_i^\uc$ intersect this line segment in one of its two
end points, depending on whether $i \in A_{p, q}^1$ or $i \in A_{p, q}^2$.
This implies that the line spanned by $S_{p, q}$ is the intersection of all
$H_\circuit$, where $\circuit \subset A_{p, q}^1 \cup \{p, q\}$ or $\circuit \subset
A_{p, q}^2 \cup \{p, q\}$. Let $D$ be $\mathfrak{A}_{p, q}$-residual and
assume that $H^i\big(X, \sh{O}(D + r D_{p, q})\big) = 0$ for all $i$ and for
infinitely many $r$. We first show that $D \in F_{\{p, q\}}$, i.e. that $c_p + c_q = -1$
for any torus invariant representative $D = \sum_{i \in \on} c_i D_i$.
Assume that $c_p + c_q > -1$. Then there exists $m \in M$ such that
$p, q \notin \mathfrak{s}^\uc(m)$. By adding sufficiently high multiples
of $D_{p, q}$ such that $D + rD_{p, q} = \sum c'_i D_i$, we can even find
such an $m$ such that $A_1 \cup A_2
\subset \mathfrak{s}^{\uc'}(m)$, hence $H^1\big(X, \sh{O}(D + r D_{p, q})\big)
\neq 0$ for large $r$ and thus $D$ is not $\mathfrak{A}_{p, q}$-residual.
If $c_p + c_q < -1$, there is an $m \in M$ with
$\{p, q\} \subset \mathfrak{s}^\uc(m)$, and by the same argument, we
get $H^2\big(X, \sh{O}(D + r D_{p, q})\big) \neq 0$ for large $r$. Hence
$c_p + c_q = -1$, i.e. $D \in F_{\{p, q\}}$.
This implies that for every $m \in M$ either $p \in \mathfrak{s}^\uc(m)$
and $q \notin \mathfrak{s}^\uc(m)$, or $q \in \mathfrak{s}^\uc(m)$
and $p \notin \mathfrak{s}^\uc(m)$.
Now assume that $D \notin \mathcal{F}_\circuit$ for
some $\circuit = \{i, j, k\} \subset A_1 \cup \{p, q\}$ such that
$\ocircuit^+ = \{i, k\}$ for some orientation.
Then there exists some $m \in M$ with $\{i, k\} \subset \mathfrak{s}^\uc(m)$
or $\{j\} \subset \mathfrak{s}^\uc(m)$. In the first case, as before we can
simply add some multiple of $D_{p, q}$ such that
$i \in \mathfrak{s}^{\uc'}(m)$ and $i \in A_2$, hence $\mathfrak{s}^{\uc'}(m)$
contains at least two $-$-intervals.
In the second case, we have either $p \notin \mathfrak{s}^\uc(m)$ or
$q \notin \mathfrak{s}(m)$, thus at least two $-$-intervals, too. Hence
$D \in \mathfrak{A}_{p, q}$ and the assertion follows.
\end{proof}

\begin{proposition}\label{nosurfacenefresiduals}
Let $X$ be a complete toric surface. Then $X$ has only a finite number of
$\mathfrak{A}_\nef$-residual divisors.
\end{proposition}

\begin{proof}
We can assume without loss of generality that $X$ is not $\mathbb{P}^2$ nor
a Hirzebruch surface.
Assume there is $D \in A_1(X)$ which is not contained in $F_\ocircuit$ for some
circuit $\circuit = \{i - 1, i, i + 1\}$ corresponding to an extremal curve on $X$.
Then there exists a chamber in the corresponding arrangement
whose signature contains $\{i - 1, i + 1\}$. To have this signature to
correspond to an acyclic subcomplex of $\hat{\Delta}$,
the rest of the signature must
contain $\on \setminus \circuit$. Now assume we have some integral vector $D_\circuit \in
H_\circuit$, then we can add a multiple of $D_\circuit$ to $D$ such that $D$ is
parallel translated to $\nef(X)$. In this process necessarily at least one
hyperplane passes the critical chamber and thus creates cohomology.

Now, $D$ might be outside of $F_\mathfrak{D}$ for some $\mathfrak{D} \in \ocircuit(L)$
not corresponding to an extremal curve. If the underlying circuit is not fibrational,
then $D$ being outside $F_\mathfrak{D}$ implies $F_\mathfrak{C}$ for
some extremal circuit $\ocircuit$. If $\mathfrak{D}$ is fibrational and
$\mathcal{D} = \{p, q\}$, then we argue as in Proposition \ref{nosurfaceantinefresiduals} that
$D$ has cohomology. If $\mathcal{D}$ is fibrational of cardinality three,
the corresponding hypersurface $H_\mathcal{D}$ is not parallel to any
nonzero face of $\nef(X)$ and there might be a finite number of
divisors lying outside $F_\mathfrak{D}$ but in the intersection of all
$F_\mathfrak{C}$, where $\mathcal{C}$ corresponds to an extremal
curve.
\end{proof}

\begin{proposition}\label{finitezeroresiduals}
Let $X$ be a complete toric surface. Then $X$ has only a finite number of
0-residual divisors.
\end{proposition}

\begin{proof}
Let us consider some vector partition function $\operatorname{VP}(L, I) :
\mathbb{O}_I \longrightarrow \N$, for $I$ such that $C_I$ does not contain a nonzero
subvector space. Let $D = \sum_{i \in \on} c_i D_i \in \Omega(L, I)$ and let $P_D$
the polytope in $M_\Q$ such that $m \in M_\Q$ is in $P_D$ iff $l_i(m) < -c_i$ for
$i \in I$ and $l_i(m) \geq -c_i$ for $i \in \on \setminus I$. For any $J \subset \on$
we denote $P_{D, J}$ the polytope defined by the same inequalities, but only for
$i \in J$. Clearly, $P_D \subset P_{D, J}$. Let $J \subset \on$ be maximal with respect
to the property that $P_{D, J}$ does not contain any lattice points. If $J \neq \on$,
then we can freely move the hyperplanes given by $l_i(m) = -c_i$ for $i \in \on
\setminus I$ such that $P_{D, J}$ remains constant and thus lattice point free.
This is equivalent to say that there exists a nonzero $D' \in \bigcap_{\circuit
\in \circuit(L_J)} H_\circuit$ and for every such $D'$ the polytope $P_{D + jD'}$
does not contain any lattice point for any $j \in \Q_{> 0}$.

Now assume that $J = \on$. This implies that the defining inequalities of $P_D$
are irredundant and thus there exists a unique maximal chamber in $C_I$ which contains
$D$ (if $I = \emptyset$ this would be the nef cone by \ref{surfacenefstratum})
and thus the combinatorial type of $P_D$ is fixed. Now, clearly, the number of
polygons of shape $P_D$ with parallel faces given by integral linear inequalities
and which do not contain a lattice point is finite.

By applying this to all (and in fact finitely many) cones $\mathbb{O}_I$ such that
$C_I$ does not contain a nontrivial subvector space of $A_\Q$, we see that there
are only finitely many divisors $D$ which are not contained in $\mathfrak{A}_\nef$ or
$\mathfrak{A}_{p, q}$.
\end{proof}

\begin{proof}[Proof of theorem \ref{surfaceclassification}.]
By \ref{surfacecores}, $\nef(X)$ and the $S_{p, q}$ are indeed the only relevant
strata, which by \ref{nosurfaceantinefresiduals} and \ref{nosurfacenefresiduals} admit
only finitely many residual elements. Hence, we are left with the $0$-residuals, of
which exist only finitely many by \ref{finitezeroresiduals}.
\end{proof}

\begin{example}
Figure \ref{hirzebruchcohom} shows the cohomology free divisors on the Hirzebruch surface
$\mathbb{F}_3$ which is given by four rays, say $l_1 = (1, 0)$, $l_2 = (0, 1)$, $l_3 = (-1, 3)$,
$l_4 = (0, -1)$ in some coordinates for $N$. In $\pic(\mathbb{F}_3) \cong \Z^2$ there are
two cones such that $H^1\big(X, \sh{O}(D)\big) \neq 0$ for every $D$ which is contained
in one of these cones. Moreover,
there is one cone such that $H^2\big(X, \sh{O}(D)\big) \neq 0$ for every $D$; its tip is
sitting at $K_{\mathbb{F}_3}$. The nef cone is indicated by the dashed lines.
\begin{figure}[htbp]
\begin{center}
\includegraphics[width=12cm]{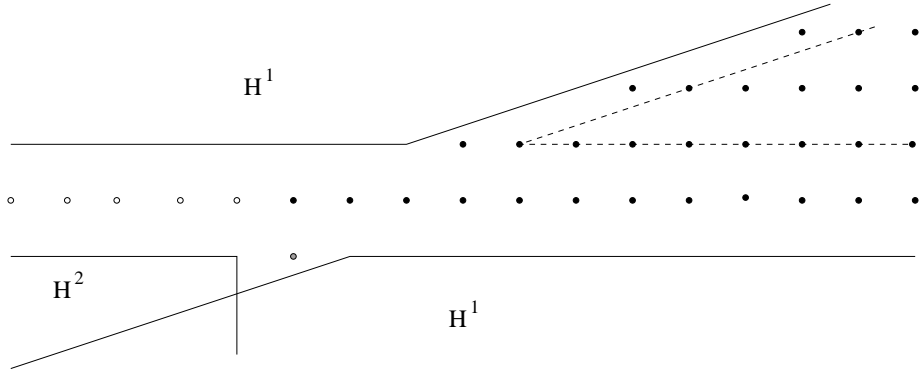}
\end{center}
\caption{Cohomology free divisors on $\mathbb{F}_3$.}\label{hirzebruchcohom}
\end{figure}
\end{example}
The picture shows the divisors contained in $\mathfrak{A}_\nef$ as black dots. The
white dots indicate the divisors in $\mathfrak{A}_{2, 4}$. There is one $0$-residual
divisor indicated by the grey dot.

The classification of smooth complete toric surfaces implies that every such surface
which is not
$\mathbb{P}^2$, has a fibrational circuit of rank one. Thus the theorem implies that
on every such surface there exist families of line bundles with vanishing
cohomology along the inverse nef cone. For a given toric surface $X$, these
families can be explicitly computed by checking for every $\circuit \subset
A_1 \cup \{p, q\}$ and every $\circuit \subset A_2 \cup \{p, q\}$, respectively,
whether the inequalities
\begin{equation*}
c_i + l_i(m)
\begin{cases}
\geq 0 & \text{ for } i \in \ocircuit^+\\
< 0 & \text{ for } i \in \ocircuit^-,
\end{cases}
\qquad
c_i + l_i(m)
\begin{cases}
\geq 0 & \text{ for } i \in -\ocircuit^+\\
< 0 & \text{ for } i \in -\ocircuit^-
\end{cases}
\end{equation*}
have solutions $m \in M$ for at least one of the two orientations
$\ocircuit$, $-\ocircuit$ of $\circuit$. This requires to deal with
$\binom{\vert A_1 \vert + 2}{3} + \binom{\vert A_2 \vert + 2}{3}$,
i.e. of order $\sim n^3$, linear inequalities. We can reduce this
number to order $\sim n^2$ as a corollary from our considerations
above:

\begin{corollary}
Let $\circuit \in A_i$ for $i = 1$ or $i = 2$. Then there exist $\{i, j\}
\subset \circuit$ such that $F_{\{p, q\}} \cap F_\ocircuit \supset F_{\{p, q\}} \cap
F_{\{i, j, p\}} \cap F_{\{i, j, q\}}$.
\end{corollary}

\begin{proof}
Assume first that there exists $m \in M$ which for the orientation $\ocircuit$
of $\circuit = \{i_1, i_2, i_3\}$ with $\ocircuit^+ = \{i_1, i_3\}$ which fulfills the
inequalities $l_{i_k}(m) + c_{i_k} \geq 0$ for $k = 1, 3$ and $l_{i_2}(m) + c_{i_2} < 0$.
This implies that $H^1\big(X, \sh{O}(D)\big) \neq 0$,
 independent of the configuration of the other hyperplanes,
as long as $c_p + c_q = -1$. It is easy to see that we can choose $i, j \in
\circuit$ such that $\{i, j, p\}$ and $\{i, j, q\}$ form circuits. We can
choose one of those such that $m$ is contained in the triangle, fulfilling
the respective inequalities, and which is not fibrational.
For the inverse orientation $-\ocircuit$, we can the same way replace one of the
elements of $\circuit$ by one of $p$, $q$. By adding a suitable positive
multiple of $D_{p, q}$, we can rearrange the hyperplanes such that
$H^1\big(X, \sh{O}(D + r D_{p, q})\big) \neq 0$.
\end{proof}

One should read the corollary the way that for any pair $i, j$ in $A_1$ or
in $A_2$, one has only to check whether a given divisor fulfills certain
inequalities for triples $\{i, j, q\}$ and $\{i, j, p\}$.
It seems that it is not possible to reduce further the number of equations
in general. However, there is a criterion which gives a good reduction
of cases for practical purposes:

\begin{corollary}\label{simplifiedconditions}
Let $X$ be a smooth and complete toric surface and
$D = \sum_{i \in \on}c_i D_i \in \mathfrak{A}_{p, q}$, then for every
$i \in A_1 \cup A_2$, we have:
\begin{equation*}
c_{i - 1} + c_{i + 1} - a_i c_i \in [-1, a_i - 1],
\end{equation*}
where the $a_i$ are the self-intersection numbers of the $D_i$.
\end{corollary}

\begin{proof}
The circuit $\circuit = \{i - 1, i, i + 1\}$ comes with the integral relation
$l_{i - 1} + l_{i + 1} + a_i l_i = 0$. So the Frobenius problem for these circuits
is trivial and we have only to consider the offset part.
\end{proof}

The following example shows that these equalities are necessary, but not sufficient
in general:

\begin{example}\label{conditioncounterexample}
We choose some coordinates on $N \cong \Z^2$ and consider the complete toric surface
defined by $8$ rays $l_1 = (0, -1)$, $l_2 = (1, -2)$, $l_3 = (1, -1)$, $l_4 = (1, 0)$,
$l_5 = (1, 1)$, $l_6 = (1, 2)$, $l_7 = (0, 1)$, $l_8 = (-1, 0)$. Then any divisor $D =
c_1 D_1 + \cdots + c_8 D_8$ with $\underline{c} = (-1, 1, 1, 0, 0, 1, 0, -k)$ for some
$k \gg 0$ has nontrivial $H^1$, though it fulfills the conditions of corollary
\ref{simplifiedconditions}.
\end{example}

An interesting and more restricting case is the additional requirement that also
$H^i\big(X, \sh{O}(-D)\big)$ $= 0$ for all $i > 0$. One may compare the following
with the classification of bundles of type $B$ in \cite{HillePerling06}.

\begin{corollary}
Let $X$ be a smooth and complete toric surface and
$D \in \mathfrak{A}_{p, q}$ such that $H^i\big(X, \sh{O}(D)\big) = H^i\big(X, \sh{O}(-D)\big) = 0$ for
all $i > 0$. Then for every $i \in A_1 \cup A_2$, we have:
\begin{equation*}
c_{i - 1} + c_{i + 1} - a_i c_i \in
\begin{cases}
\{\pm 1, 0\} & \text{ if } a_i < -1 \\
\{-1, 0\} & \text{ if } a_i = -1,
\end{cases}
\end{equation*}
where the $a_i$ are the self-intersection numbers of the $D_i$.
\end{corollary}

\begin{proof}
For $-D$, we have $c_p + c_q = 1$. Assume that there exists a circuit
circuit $\circuit$ with orientation $\ocircuit$ and $\ocircuit^+ = \{i, j\}$
and $\ocircuit^- = \{ k\}$, and morover, some lattice point
$m$ such that $\mathfrak{s}^\uc(m) \cap \circuit = \ocircuit^-$. Then we
get $\mathfrak{s}^{-\uc}(-m) \cap \circuit = \ocircuit^+$.
this implies that $H^1\big(X, \sh{O}(-D)\big) \neq 0$. This implies the
restriction $c_{i - 1} + c_{i + 1} - a_i c_i \in [-1, \min\{1, a_i - 1\}]$.
\end{proof}

Note that example \ref{conditioncounterexample} also fulfills these more restrictive
conditions.

\subsection{Maximal Cohen-Macaulay Modules of Rank One}
\label{mcmsection}

The classification of maximal Cohen-Macaulay modules can sometimes be related
to resolution of singularities, the most famous example for this being the
McKay correspondence in the case of certain surface singularities
(\cite{GonzalezSprinbergVerdier}, \cite{ArtinVerdier}, see also
\cite{EsnaultKnoerrer}). In the toric case, in general one cannot expect to
arrive at such a nice picture, as there does not exist a
canonical way to construct resolutions. However, there is a natural set of
preferred partial resolutions, which is parameterized by the secondary fan.

Let $X$ be a $d$-dimensional affine toric variety whose associated convex
polyhedral cone $\sigma$ has dimension $d$. Denote $x \in X$ torus fixed
point. For any Weil divisor $D$ on $X$, the sheaf $\sh{O}_X(D)$ is MCM if
and only if $H^i_x\big(X, \sh{O}_X(D)\big) $ for all $i < d$. It was shown
in \cite{BrunsGubeladze} (see also \cite{BrunsGubeladze2}) that there exists
only a finite number of such modules.

 Now let
$\tilde{X}$ be a toric variety given by some triangulation of
$\sigma$. The natural map $\pi: \tilde{X} \longrightarrow X$ is a partial
resolution of the singularities of $X$ which is an isomorphism in codimension
two and has at most quotient singularities.
In particular, the map of fans is induced by the identity on $N$ and,
in turn, induces a bijection on the set of torus invariant Weil divisors.
This bijection induces a natural isomorphism $\pi^{-1}: A_{d - 1}(X) \longrightarrow
A_{d - 1} (\tilde{X})$ which can be represented by the identity morphism on the
invariant divisor group $\Z^n$. This allows us to identify a
torus invariant divisor $D$ on $X$ with its strict transform $\pi^{-1} D$ on $\tilde{X}$.
Moreover, there are the natural isomorphisms
\begin{equation*}
\pi_*\sh{O}_{\tilde{X}}(\pi^{-1}D) \cong \sh{O}_X(D)\quad \text{ and } \quad
\sh{O}_{\tilde{X}}(\pi^{-1}D) \cong \big(\pi^*\sh{O}_{X}(D)\big)\check{\ }\check{\ }.
\end{equation*}

Our aim is to compare local cohomology and global cohomology, i.e.
$H^i_x\big(X, \sh{O}_X(D)\big)$ and $H^i\big(\tilde{X}, \sh{O}_{\tilde{X}}(D)\big)$.
In general, we have the following easy statement about general (i.e. non-regular)
triangulations:
\begin{theorem}
Let $X$ be an affine toric variety of dimension $d$ and $D \in A_{d - 1}(X)$. If
$D$ is $0$-essential, then
$R^i\pi_* \sh{O}_{\tilde{X}}(\pi^*D)
= 0$ for every triangulation $\pi: \tilde{X} \longrightarrow X$.
\end{theorem}

\begin{proof}
If $D$ is $0$-essential, then it is contained in the intersection of all $F_\circuit$,
where $\circuit \in \circuit(L)$, thus it represents a cohomology-free divisor.
\end{proof}

Note that the statement does hold for any triangulation and not only for regular
triangulations.
We have a refined statement for affine toric varieties whose associated cone
$\sigma$ has simplicial facets:

\begin{theorem}\label{mcmtheorem}
Let $X$ be a $d$-dimensional affine toric variety whose associated cone $\sigma$
has simplicial facets and let $D \in A_{d - 1}(X)$. If
$R^i\pi_* \sh{O}_{\tilde{X}}(\pi^*D) = 0$ for every regular triangulation
$\pi: \tilde{X} \longrightarrow X$  then $\sh{O}_X(D)$ is MCM. For $d = 3$
the converse is also true.
\end{theorem}

\begin{proof}
Recall that $H^i_x\big(X, \sh{O}(D)\big)_m =$ $H^{i - 2}(\hat{\sigma}_{V, m}; k)$
for some $m \in M$ and $D \in A$.
We are going to show that for every subset $I \subsetneq \on$
there exists a regular triangulation $\tilde{\Delta}$ of $\sigma$ such that the
simplicial complexes $\hat{\sigma}_{V, I}$ and $\tilde{\Delta}_I$ coincide.
This implies that if $H^i_x\big(X, \sh{O}_X(D)\big)_m \neq 0$ for some
$m \in M$, then also $H^{i + 1}\big(\tilde{X}, \sh{O}_{\tilde{X}}(D)\big)_m \neq 0$,
i.e. if
$\sh{O}_X(D)$ is not MCM, then $H^i\big(\tilde{X}, \sh{O}_{\tilde{X}}(D)\big) \neq
0$ for some $i > 0$.

For given $I \subset \on$ we get such a triangulation as follows. Let $i \in \on
\setminus I$ and consider the dual cone $\check{\sigma}$. Denote $\rho_i := \Q_{\geq 0}
l_i$ and recall that $\check{\rho}_i$ is a halfspace which contains $\check{\sigma}$
and which defines a facet of $\check{\sigma}$ given by $\rho^\bot \cap \check{\sigma}$.
Now we move $\check{\rho}_i$ to $m + \check{\rho}$, where $l_i(m) > 0$.
So we obtain a new polytope $P := \check{\sigma} \cap (m + \check{\rho})$. As
$\rho^\bot$  is not parallel to any face of $\check{\sigma}$, the hyperplane
$m + \rho^\bot$ intersects every face of $\check{\sigma}$. This way the inner
normal fan of $P$ is
a triangulation $\tilde{\Delta}$ of $\sigma$ which has the property that
every maximal cone is spanned by $\rho_i$ and some facet of $\sigma$.
This implies $\tilde{\Delta}_I = \hat{\sigma}_{V, I}$ and the first assertion follows.

For $d = 3$, a sheaf $\sh{O}(D)$ is MCM iff $H^2_x\big(X, \sh{O}(D)\big) = 0$,
i.e. $H^0(\sigma_{V, m}; k) = 0$ for every $m \in M$. The latter is only possible
if $\sigma_{V, m}$ represents an interval on $S^1$. To compare
this with $H^2\big(\tilde{X}, \sh{O}(D)\big)$ for some regular triangulation
$\tilde{X}$, we must show that $H^1(\tilde{\Delta}_m; k) = 0$ for the corresponding
complex $\tilde{\Delta}_m$. To see this, we consider some cross-section
$\sigma \cap H$, where $H \subset N \otimes_Z \R$ is some hyperplane which intersects
$\sigma$ nontrivially and is not parallel to any of its faces. Then this cross-section
can be considered as a planar polygon and $\sigma_{V, m}$ as some connected sequence
of faces of this polygon. Now with respect to the triangulation $\tilde{\Delta}$ of
this polygon, we can consider two vertices $p, q \in \sigma_{V, m}$ which are
connected by a line belonging to the triangulation and going through the interior
of the polygon. We assume that $p$ and $q$ have maximal distance in $\sigma_{V, m}$
with this
property. Then it is easy to see that the triangulation of $\sigma$ induces a
triangulation of the convex hull of the line segments connecting $p$ and $q$. Then
$\tilde{\Delta}_m$ is just the union of this convex hulls with respect all such pairs
$p, q$ and the remaining line
segments and thus has trivial topology. Hence $H^2_x\big(X, \sh{O}(D)\big) = 0$
implies $H^2\big(\tilde{X}, \sh{O}(D)\big) = 0$ for every triangulation
$\tilde{\Delta}$ of $\sigma$.
\end{proof}

\begin{example}\label{mcmexampleone}
Consider the $3$-dimensional cone spanned over the primitive vectors $l_1 = (1,0,1)$,
$l_2 = (0, 1, 1)$, $l_3 = (-1, 0, 1)$, $l_4 = (-1, -1, 1)$, $l_5 = (1, -1, 1)$.
The corresponding toric variety $X$ is Gorenstein and its divisor class group is
torsion free. For $A_2(X) \cong \Z^2$ we choose the basis $D_1 + D_2 + D_5$, $D_5$.
In this basis, the Gale duals of the $l_i$ are $D_1 = (-1, -1)$, $D_2 = (2, 0)$,
$D_3 = (-3, 1)$, $D_4 = (2, -1)$, $D_5 = (0, 1)$. Figure \ref{mcmpic} shows the
set of MCM modules in $A_2(X)$ which are indicated by circles which are supposed to
sit on
the lattice $A_2(X) \cong \Z^2$. The picture also indicates the cones $C_I$ with
vertices $-e_I$, where $I \in \{\{1, 3\}, \{1, 4\}, \{2, 4\}, \{2, 5\}, \{3, 5\},
\{1, 2, 5\}, \{1, 3, 4\}, \{1, 3, 5\}, \{2, 3, 5\}, \{2, 4, 5\}\}$.
Note that the picture has a reflection symmetry, due to the fact that $X$ is
Gorenstein. Altogether,
there are $19$ MCM modules of rank one, all of which are $0$-essential.
For $\circuit = \{l_1, l_3, l_4, l_5\}$, the group $A_2(X)_\circuit \cong
\Z \oplus \Z / 2 \Z$ has torsion.
The two white circles indicate modules are contained in the $\Q$-hyperplanes
$D_1 + D_4 + H_\circuit$ and $D_2 + D_3 + D_5 + H_\circuit$, respectively,
but not in the sets $D_1 + D_4 +
Z_\circuit$ and $D_2 + D_3 + D_5 + Z_\circuit$, respectively.
Some of the $\mathbb{O}_I$ are not saturated; however, every divisor which is
contained in some $(-e_I + C_I) \cap \Omega(L, I)$ is also contained in some
$\mathbb{O}_{I'} \setminus
\Omega(L, I')$ for some other
$I' \neq I$. So for this example, the Frobenius arrangement gives a full description
of MCM modules of rank one.
\begin{figure}[htbp]
\begin{center}
\includegraphics[width=10cm]{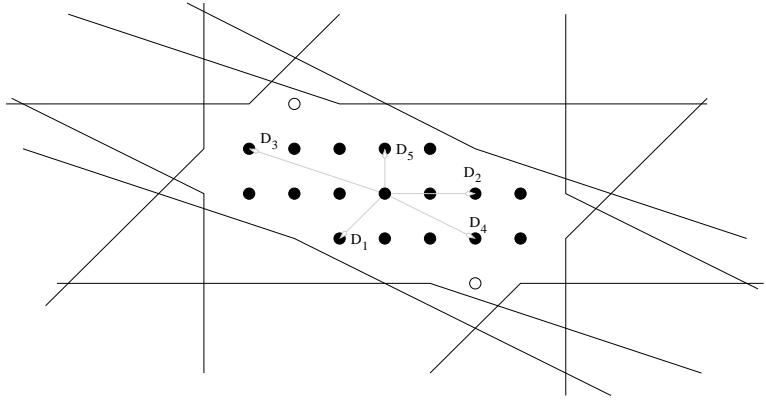}
\end{center}
\caption{The 19 MCM modules of example \ref{mcmexampleone}.}\label{mcmpic}
\end{figure}
\end{example}

\begin{example}\label{mcmexampletwo}
To give a counterexample to the reverse direction of theorem \ref{mcmtheorem}
for $d > 3$, we consider the four-dimensional cone spanned over the primitive
vectors $l_1 =
(0, -1, -1, 1)$, $l_2 = (-1, 0, 1, 1)$, $l_3 = (0, 1, 0, 1)$, $l_4 = (-1, 0, 0, 1)$,
$l_5 = (-1, -1, 0, 1)$, $l_6 = (1, 0, 0, 1)$. The corresponding variety $X$ has
31 MCM modules of rank one which are shown in figure \ref{mcm2pic}. Here, with
basis $D_1$ and $D_6$, we have $D_1 = (1, 0)$, $D_2 = (1, 0)$, $D_3 = (-1, -2)$,
$D_4 = (3, 1)$, $D_5 = (-2, -2)$, $D_6 = (0, 1)$. There are six cohomology cones
corresponding to $I \in \big\{\{1, 2\}, \{3, 5\}, \{4, 6\}, \{1, 2, 3, 5\},
\{1, 2, 4, 6\}, \{3, 4, 5, 6\}\big\}$.

\begin{figure}[htbp]
\begin{center}
\includegraphics[width=10cm]{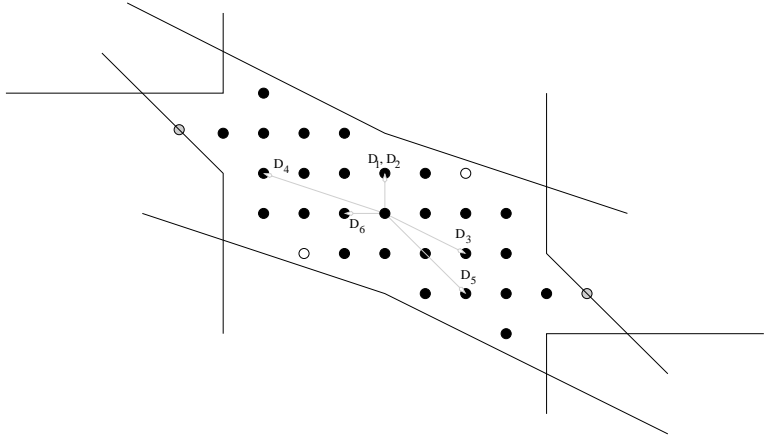}
\caption{The 31 MCM modules of example \ref{mcmexampletwo}.}\label{mcm2pic}
\end{center}
\end{figure}

The example features two modules
which are not $0$-essential, indicated by the grey dots sitting on the boundary
of the cones $-e_I + C_I$, where $I \in \big\{\{4, 6\}, \{1, 2, 3, 5\}\big\}$.
The white dots denote MCM divisors $D, -D$ such that there exists a triangulation of
the cone of $X$ such that on the associated variety $\tilde{X}$ we have
$H^i\big(\tilde{X}, \sh{O}(\pm D)\big) \neq 0$ for some $i > 0$. Namely, we consider
the triangulation which is given by the maximal cones spanned by
$\{l_1, l_2, l_4, l_5\}$, $\{l_1, l_2, l_4, l_6\}$, $\{l_1, l_2, l_5, l_6\}$,
$\{l_1, l_3, l_4, l_6\}$, $\{l_2, l_3, l_4, l_6\}$. Figure \ref{counterpic} indicates
the two-dimensional faces of this triangulation via a three-dimensional cross-section
of the cone.

\begin{figure}[htbp]
\begin{center}
\includegraphics[height=4cm]{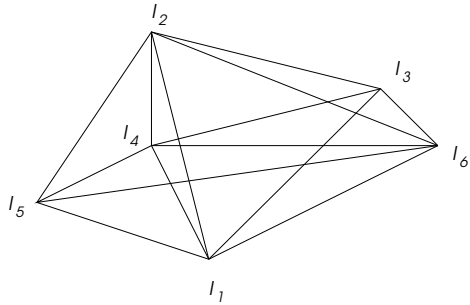}\label{counterpic}
\caption{The triangulation for $\tilde{X}$ in example \ref{mcmexampletwo}.}
\end{center}
\end{figure}

We find that we have six cohomology cones corresponding to $I \in \big\{
\{1, 2\}, \{3, 5\}$, $\{1, 2, 3\}$, $\{4, 5, 6\}$, $\{1, 2, 3, 5\},
\{3, 4, 5, 6\}\big\}$. In particular, we have non-vanishing $H^1$ for the points
$-D_1 - D_2 - D_3$ and for $-D_4 - D_5 - D_6$, which correspond to $D$ and $-D$.
\end{example}

\newcommand{\etalchar}[1]{$^{#1}$}

\end{document}